\documentclass[12pt]{amsart}       
\usepackage{txfonts}
\usepackage{amssymb}
\usepackage{eucal}
\usepackage{graphicx}
\usepackage{amsmath}
\usepackage{amscd}
\usepackage[all]{xy}           
\usepackage{amsfonts,latexsym}
\usepackage{xspace}
\usepackage{epsfig}
\usepackage{float}
\usepackage{color}
\usepackage{fancybox}
\usepackage{colordvi}
\usepackage{multicol}
\usepackage{colordvi}
\usepackage{ifpdf}
\ifpdf
  \usepackage[colorlinks,final,backref=page,hyperindex]{hyperref}
\else
  \usepackage[colorlinks,final,backref=page,hyperindex,hypertex]{hyperref}
\fi
\usepackage[active]{srcltx} 

\usepackage{tikz}

\usepackage{graphicx}

\newtheorem{theorem}{Theorem}[section]
\newtheorem{prop}[theorem]{Proposition}

\newtheorem{coro}[theorem]{Corollary}
\newtheorem{prop-def}{Proposition-Definition}[section]
\newtheorem{coro-def}{Corollary-Definition}[section]

\theoremstyle{definition}
\newtheorem{defn}[theorem]{Definition}
\newtheorem{remark}[theorem]{Remark}

\newtheorem{exam}[theorem]{Example}

\newcommand{\nc}{\newcommand}
\nc{\tred}[1]{\textcolor{red}{#1}}
\nc{\tblue}[1]{\textcolor{blue}{#1}}
\nc{\tgreen}[1]{\textcolor{green}{#1}}
\nc{\tpurple}[1]{\textcolor{purple}{#1}}
\nc{\btred}[1]{\textcolor{red}{\bf #1}}
\nc{\btblue}[1]{\textcolor{blue}{\bf #1}}
\nc{\btgreen}[1]{\textcolor{green}{\bf #1}}
\nc{\btpurple}[1]{\textcolor{purple}{\bf #1}}
\nc{\NN}{{\mathbb N}}
\nc{\ncsha}{{\mbox{\cyr X}^{\mathrm NC}}} \nc{\ncshao}{{\mbox{\cyr
X}^{\mathrm NC}_0}}

\renewcommand{\frak}{\mathfrak}

\renewcommand{\textbf}[1]{}

\newcommand{\delete}[1]{}

\nc{\mlabel}[1]{\label{#1}}
\nc{\mcite}[1]{\cite{#1}}
\nc{\mref}[1]{\ref{#1}}
\nc{\meqref}[1]{\eqref{#1}}
\nc{\mbibitem}[1]{\bibitem{#1}}

\delete{
\nc{\mlabel}[1]{\label{#1}{\hfill \hspace{1cm}{\bf{{\ }\hfill(#1)}}}}
\nc{\mcite}[1]{\cite{#1}{{\bf{{\ }(#1)}}}}
\nc{\mref}[1]{\ref{#1}{{\bf{{\ }(#1)}}}}
\nc{\meqref}[1]{\eqref{#1}{{\bf{{\ }(#1)}}}}
\nc{\mbibitem}[1]{\bibitem[\bf #1]{#1}}
}

\nc{\Shu}{\mathrm{Sh}}
\nc{\tot}{\big|}
\nc{\AAO}{{A,\Omega}}
\nc{\msha}{\sha_\Omega}
\nc{\mrsha}{\sha_{\Omega,F}}

\nc{\opa}{\ast} \nc{\opb}{\odot} \nc{\op}{\bullet} \nc{\pa}{\frakL}
\nc{\arr}{\rightarrow} \nc{\lu}[1]{(#1)} \nc{\mult}{\mrm{mult}}
\nc{\diff}{\mathfrak{Diff}}
\nc{\opc}{\sharp}\nc{\opd}{\natural}
\nc{\ope}{\circ}
\nc{\dpt}{\mathrm{d}}
\nc{\hck}{H_{RT}}
\nc{\vdf}{\calf}
\nc{\ldf}{\calf_\ell}
\nc{\hlf}{H_\ell}
\nc{\onek}{\mathbf{1}_\bfk}
\nc{\tforall}{\, \text{ for }\, }
\nc{\qforall}{\quad \text{for all }}

\nc{\mrba}{MRBA\xspace}
\nc{\Mrba}{MRBA\xspace}
\nc{\mrbas}{MRBAs\xspace}
\nc{\Mrbas}{MRBAs\xspace}
\nc{\match}{matching\xspace}
\nc{\Match}{Matching\xspace}
\nc{\Mza}{Matching Zinbiel algebra\xspace}
\nc{\Mzas}{Matching Zinbiel algebras\xspace}
\nc{\mza}{matching Zinbiel algebra\xspace}
\nc{\mzas}{matching Zinbiel algebras\xspace}

\nc{\za}{Zinbiel algebra\xspace}
\nc{\paybe}{polarized associative Yang-Baxter equation\xspace}
\nc{\Paybe}{Polarized associative Yang-Baxter equation\xspace}
\nc{\cpaybe}{PAYBE}

\nc{\diam}{alternating\xspace}
\nc{\Diam}{Alternating\xspace}
\nc{\cdiam}{canonical alternating\xspace}
\nc{\Cdiam}{Canonical alternating\xspace}
\nc{\AW}{\mathcal{A}}
\nc{\rba}{Rota-Baxter algebra\xspace}
\nc{\rbas}{Rota-Baxter algebras\xspace}

\nc{\ari}{\mathrm{ar}}
\nc{\lef}{\mathrm{lef}}
\nc{\Sh}{\mathrm{ST}}
\nc{\Cr}{\mathrm{Cr}}
\nc{\st}{{Schr\"oder tree}\xspace}
\nc{\sts}{{Schr\"oder trees}\xspace}
\nc{\vertset}{\Omega} 
\nc{\assop}{\quad \begin{picture}(5,5)(0,0)
\line(-1,1){10}
\put(-2.2,-2.2){$\bullet$}
\line(0,-1){10}\line(1,1){10}
\end{picture} \quad \smallskip}

\nc{\operator}{\begin{picture}(5,5)(0,0)
\line(0,-1){6}
\put(-2.6,-1.8){$\bullet$}
\line(0,1){9}
\end{picture}}

\nc{\idx}{\begin{picture}(6,6)(-3,-3)
\put(0,0){\line(0,1){6}}
\put(0,0){\line(0,-1){6}}
\end{picture}}

\nc{\pb}{{\mathrm{pb}}}
\nc{\Lf}{{\mathrm{Lf}}}

\nc{\lft}{{left tree}\xspace}
\nc{\lfts}{{left trees}\xspace}

\nc{\fat}{{fundamental averaging tree}\xspace}

\nc{\fats}{{fundamental averaging trees}\xspace}
\nc{\avt}{\mathrm{Avt}}

\nc{\rass}{{\mathit{RAss}}}

\nc{\aass}{{\mathit{AAss}}}

\nc{\twovec}[2]{{\textstyle \left[#1\atop #2\right]}}

\nc{\vin}{{\mathrm Vin}}    
\nc{\lin}{{\mathrm Lin}}    
\nc{\inv}{\mathrm{I}n}
\nc{\gensp}{V} 
\nc{\genbas}{\mathcal{V}} 
\nc{\bvp}{V_P}     
\nc{\gop}{{\,\omega\,}}     

\nc{\bin}[2]{ (_{\stackrel{\scs{#1}}{\scs{#2}}})}  
\nc{\binc}[2]{ \left (\!\! \begin{array}{c} \scs{#1}\\
    \scs{#2} \end{array}\!\! \right )}  
\nc{\bincc}[2]{  \left ( {\scs{#1} \atop
    \vspace{-1cm}\scs{#2}} \right )}  
\nc{\bs}{\bar{S}} \nc{\cosum}{\sqsubset} \nc{\la}{\longrightarrow}
\nc{\rar}{\rightarrow} \nc{\dar}{\downarrow} \nc{\dprod}{**}
\nc{\dap}[1]{\downarrow \rlap{$\scriptstyle{#1}$}}
\nc{\md}{\mathrm{dth}} \nc{\uap}[1]{\uparrow
\rlap{$\scriptstyle{#1}$}} \nc{\defeq}{\stackrel{\rm def}{=}}
\nc{\disp}[1]{\displaystyle{#1}} \nc{\dotcup}{\
\displaystyle{\bigcup^\bullet}\ } \nc{\gzeta}{\bar{\zeta}}
\nc{\hcm}{\ \hat{,}\ } \nc{\hts}{\hat{\otimes}}
\nc{\barot}{{\otimes}} \nc{\free}[1]{\bar{#1}}
\nc{\uni}[1]{\tilde{#1}} \nc{\hcirc}{\hat{\circ}} \nc{\lleft}{[}
\nc{\lright}{]} \nc{\lc}{\lfloor} \nc{\rc}{\rfloor}
\nc{\curlyl}{\left \{ \begin{array}{c} {} \\ {} \end{array}
    \right .  \!\!\!\!\!\!\!}
\nc{\curlyr}{ \!\!\!\!\!\!\!
    \left . \begin{array}{c} {} \\ {} \end{array}
    \right \} }
\nc{\longmid}{\left | \begin{array}{c} {} \\ {} \end{array}
    \right . \!\!\!\!\!\!\!}
\nc{\onetree}{\bullet} \nc{\ora}[1]{\stackrel{#1}{\rar}}
\nc{\ola}[1]{\stackrel{#1}{\la}}
\nc{\ot}{\otimes} \nc{\mot}{{{\boxtimes\,}}}
\nc{\otm}{\overline{\boxtimes}} \nc{\sprod}{\bullet}
\nc{\scs}[1]{\scriptstyle{#1}} \nc{\mrm}[1]{{\rm #1}}
\nc{\margin}[1]{\marginpar{\rm #1}}   
\nc{\dirlim}{\displaystyle{\lim_{\longrightarrow}}\,}
\nc{\invlim}{\displaystyle{\lim_{\longleftarrow}}\,}
\nc{\mvp}{\vspace{0.3cm}} \nc{\tk}{^{(k)}} \nc{\tp}{^\prime}
\nc{\ttp}{^{\prime\prime}} \nc{\svp}{\vspace{2cm}}
\nc{\vp}{\vspace{8cm}} \nc{\proofbegin}{\noindent{\bf Proof: }}
\nc{\proofend}{$\blacksquare$ \vspace{0.3cm}}
\nc{\modg}[1]{\!<\!\!{#1}\!\!>}
\nc{\intg}[1]{F_C(#1)} \nc{\lmodg}{\!
<\!\!} \nc{\rmodg}{\!\!>\!}
\nc{\cpi}{\widehat{\Pi}}
\nc{\sha}{{\mbox{\cyr X}}}  

\newfont{\scyr}{wncyr10 scaled 550}
\nc{\ssha}{\mbox{\bf \scyr X}}
\nc{\shap}{\,{\mbox{\cyrs X}}\,} 
\nc{\shpr}{\diamond}    
\nc{\shp}{\ast} \nc{\shplus}{\shpr^+}
\nc{\shprc}{\shpr_c}    
\nc{\msh}{\ast} \nc{\zprod}{m_0} \nc{\oprod}{m_1}
\nc{\vep}{\epsilon} \nc{\labs}{\mid\!} \nc{\rabs}{\!\mid}
\nc{\sqmon}[1]{\langle #1\rangle}

\nc{\mmbox}[1]{\mbox{\ #1\ }} \nc{\dep}{\mrm{dep}} \nc{\fp}{\mrm{FP}}
\nc{\rchar}{\mrm{char}} \nc{\End}{\mrm{End}} \nc{\Fil}{\mrm{Fil}}
\nc{\Mor}{Mor\xspace} \nc{\gmzvs}{gMZV\xspace}
\nc{\gmzv}{gMZV\xspace} \nc{\mzv}{MZV\xspace}
\nc{\mzvs}{MZVs\xspace} \nc{\Hom}{\mrm{Hom}} \nc{\id}{\mrm{id}}
\nc{\im}{\mrm{im}} \nc{\incl}{\mrm{incl}} \nc{\map}{\mrm{Map}}
\nc{\mchar}{\rm char} \nc{\nz}{\rm NZ} \nc{\supp}{\mathrm Supp}

\nc{\Alg}{\mathbf{Alg}} \nc{\Bax}{\mathbf{Bax}} \nc{\bff}{\mathbf f}
\nc{\bfk}{{\bf k}} \nc{\bfone}{{\bf 1}} \nc{\bfx}{\mathbf x}
\nc{\bfy}{\mathbf y}
\nc{\base}[1]{\bfone^{\otimes ({#1}+1)}} 
\nc{\Cat}{\mathbf{Cat}}

\nc{\detail}{\marginpar{\bf More detail}
    \noindent{\bf Need more detail!}
    \svp}
\nc{\Int}{\mathbf{Int}} \nc{\Mon}{\mathbf{Mon}}
\nc{\rbtm}{{shuffle }} \nc{\rbto}{{Rota-Baxter }}
\nc{\remarks}{\noindent{\bf Remarks: }} \nc{\Rings}{\mathbf{Rings}}
\nc{\Sets}{\mathbf{Sets}} \nc{\wtot}{\widetilde{\odot}}
\nc{\wast}{\widetilde{\ast}} \nc{\bodot}{\bar{\odot}}
\nc{\bast}{\bar{\ast}} \nc{\hodot}[1]{\odot^{#1}}
\nc{\hast}[1]{\ast^{#1}} \nc{\mal}{\mathcal{O}}
\nc{\tet}{\tilde{\ast}} \nc{\teot}{\tilde{\odot}}
\nc{\oex}{\overline{x}} \nc{\oey}{\overline{y}}
\nc{\oez}{\overline{z}} \nc{\oef}{\overline{f}}
\nc{\oea}{\overline{a}} \nc{\oeb}{\overline{b}}
\nc{\weast}[1]{\widetilde{\ast}^{#1}}
\nc{\weodot}[1]{\widetilde{\odot}^{#1}} \nc{\hstar}[1]{\star^{#1}}
\nc{\lae}{\langle} \nc{\rae}{\rangle}
\nc{\lf}{\lfloor}
\nc{\rf}{\rfloor}

\nc{\QQ}{{\mathbb Q}}
\nc{\RR}{{\mathbb R}} \nc{\ZZ}{{\mathbb Z}}

\nc{\cala}{{\mathcal A}} \nc{\calb}{{\mathcal B}}
\nc{\calc}{{\mathcal C}}
\nc{\cald}{{\mathcal D}} \nc{\cale}{{\mathcal E}}
\nc{\calf}{{\mathcal F}} \nc{\calg}{{\mathcal G}}
\nc{\calh}{{\mathcal H}} \nc{\cali}{{\mathcal I}}
\nc{\call}{{\mathcal L}} \nc{\calm}{{\mathcal M}}
\nc{\caln}{{\mathcal N}} \nc{\calo}{{\mathcal O}}
\nc{\calp}{{\mathcal P}} \nc{\calr}{{\mathcal R}}
\nc{\cals}{{\mathcal S}} \nc{\calt}{{\mathcal T}}
\nc{\calu}{{\mathcal U}} \nc{\calw}{{\mathcal W}} \nc{\calk}{{\mathcal K}}
\nc{\calx}{{\mathcal X}} \nc{\CA}{\mathcal{A}}

\nc{\fraka}{{\mathfrak a}} \nc{\frakA}{{\mathfrak A}}
\nc{\frakb}{{\mathfrak b}} \nc{\frakB}{{\mathfrak B}}
\nc{\frakc}{{\mathfrak c}}
\nc{\frakD}{{\mathfrak D}} \nc{\frakF}{\mathfrak{F}}
\nc{\frakf}{{\mathfrak f}} \nc{\frakg}{{\mathfrak g}}
\nc{\frakH}{{\mathfrak H}} \nc{\frakL}{{\mathfrak L}}
\nc{\frakM}{{\mathfrak M}} \nc{\bfrakM}{\overline{\frakM}}
\nc{\frakm}{{\mathfrak m}} \nc{\frakP}{{\mathfrak P}}
\nc{\frakN}{{\mathfrak N}} \nc{\frakp}{{\mathfrak p}}
\nc{\frakS}{{\mathfrak S}} \nc{\frakT}{\mathfrak{T}}
\nc{\frakX}{{\mathfrak X}}
\nc{\BS}{\mathbb{S
}}

\font\cyr=wncyr10 \font\cyrs=wncyr7
\nc{\li}[1]{\textcolor{red}{#1}}
\nc{\lir}[1]{\textcolor{red}{Li:#1}}
\nc{\yi}[1]{\textcolor{blue}{Yi: #1}}
\nc{\xing}[1]{\textcolor{purple}{Xing:#1}}
\nc{\revise}[1]{\textcolor{red}{#1}}

\nc{\ID}{{\rm I}}\nc{\lbar}[1]{\overline{#1}}\nc{\bre}{{\rm bre}}
\nc{\sd}{\cals}\nc{\rb}{\rm RB}\nc{\A}{\rm A}\nc{\LL}{\rm L}\nc{\tx}{\tilde{X}}
\nc{\col}{\Delta_{RT}}\nc{\mul}{m_{RT}}\nc{\ul}{u_{RT}}\nc{\epl}{\varepsilon_{RT}}
\nc{\hl}{H_{RT}}\nc{\arro}[1]{#1}\nc{\px}{P_{\tx}}\nc{\pw}{P_{\mathfrak{w}}}\nc{\pl}{B_\omega^+}
\nc{\pp}{\pl}\nc{\ppp}[1]{B^+(#1)}\nc{\dw}{\diamond_{\mathfrak{w}}}\nc{\dl}{\diamond_{\rm \ell}}
\nc{\ncshaw}{\sha^{{\rm NC}}_{\Omega}}\nc{\ncshal}{\sha^{{\rm NC}}_{{\rm RT}}}
\nc{\ver}{\rm V}\nc{\ld}{l}\nc{\del}{\Delta_{{\rm \ell}}}\nc{\epsl}{\epsilon_{{\rm \ell}}}
\nc{\uul}{u_{{\rm \ell}}}\nc{\oneh}{\mathbf{1}}\nc{\onew}{\mathbf{1}}
\nc{\etree}{1} \nc{\conc}{m_{RT}}
\nc{\hrtb}{H_{RT}(X\sqcup\Omega)} \nc{\hrts}{H_{RT}(X, \Omega)}\nc{\rts}{\mathcal{T}(X, \Omega)}\nc{\rfs}{\mathcal{F}(X, \Omega)} \nc{\ncshall}{\sha^{{\rm NC}}_{{\rm RT}}} \nc{\ldl}{\leq_{\mathrm{dl}}} \nc{\pla}{B_{\alpha}^{+}} \nc{\plb}{B_{\beta}^{+}}

\nc{\bim}[1]{#1}  \nc{\shaop}{\sha_{\Omega}^{+}}  \nc{\shao}{\sha_{\Omega}}
\nc{\bbim}[2]{#1 #2} \nc{\bbbim}[2]{#1,\, #2} \nc{\RBF}{{\rm MRB}}
\nc{\frbf}{F_{\RBF}} \nc{\shaf}{\ssha_{\tiny{\Omega}}} \nc{\sham}{\diamond}

\nc{\dnx}{\Delta_n A} \nc{\dx}{\Delta A} \nc{\dgp}{{\rm deg_{P}}}
\nc{\dgt}{{\rm deg_{T}}} \nc{\dg}{{\rm deg}} \nc{\ida}{ID($A$)} \nc{\tu}{\tilde{u}} \nc{\tv}{\tilde{v}}
 \nc{\fbase}{\calb} \nc{\LF}{\mathrm{RF}} \nc{\FFA}{\mathrm{LF}} \nc{\irr}{\mathrm{Irr}}
 \nc{\result}{\bfk\mathrm{Irr}(S_n)}  \nc{\I}{I_{\mathrm{ID},n}^0}
 \nc{\nrs}{\calr_n^\star} \nc{\ii}{\mathrm{I}} \nc{\iii}{\mathrm{II}}
\nc{\intl}{{\rm int}}\nc{\ws}[1]{{#1}}\nc{\deleted}[1]{\delete{#1}}\nc{\plas}{placements\xspace}

\nc{\Id}{\mathrm{Id}} \nc{\Irr}{\mathrm{Irr}}

\nc{\tos}{totally ordered set} \nc{\nes}{nonempty set}
\nc{\rsha}{\sha^{\rm rel}}
\nc{\lra}{\longrightarrow}

\nc{\rel}{\mathrm{\text{rel}}}

\nc{\rsham}{\sham}
\nc{\bo}{\tau}

\nc{\sshao}{~\ssha_{\Omega}~}
\nc{\resham}{\diamond_{\Omega}^{\mathrm{rel}}}

\nc{\fraku}{\frak{u}}
\nc{\frakv}{\frak{v}}
\nc{\frabu}{\bar{\mathfrak{u}}}
\nc{\frabv}{\bar{\mathfrak{v}}}
\nc{\frabw}{\bar{\mathfrak{w}}}

\nc{\udl}[1]{\underline{#1}}

\nc{\Po}{(P_\omega)_{\omega\in \Omega}}
\nc{\basf}{F} \nc{\frakhat}[1]{\frak{#1'}}
\nc{\cten}[2]{\left[\begin{array}{c}#1 \\#2 \\ \end{array}\right]}

\begin{document}

\title[Commutative matching Rota-Baxter algebras and shuffle products with decorations]{
Commutative matching Rota-Baxter operators, shuffle products with decorations and matching Zinbiel algebras}
%
\author{Xing Gao}
\address{School of Mathematics and Statistics, Key Laboratory of Applied Mathematics and Complex Systems, Lanzhou University, Lanzhou, Gansu 730000, P.\,R. China;
and College of Science, Northwest A\& F University, Yangling, Shaanxi, 712100, China}
\email{gaoxing@lzu.edu.cn}

\author{Li Guo}
\address{Department of Mathematics and Computer Science, Rutgers University, Newark, NJ 07102, USA}
\email{liguo@rutgers.edu}

\author{Yi Zhang }
\address{School of Mathematics and Statistics,
Nanjing University of Information Science \& Technology, Nanjing, Jiangsu 210044, P.\,R. China}
\email{zhangy2016@nuist.edu.cn}

\date{\today}
\begin{abstract}
The Rota-Baxter algebra and the shuffle product are both algebraic structures arising from integral operators and integral equations. Free commutative Rota-Baxter algebras provide an algebraic framework for integral equations with the simple Riemann integral operator. The Zinbiel algebras form a category in which the shuffle product algebra is  the free object. Motivated by algebraic structures underlying integral equations involving multiple integral operators and kernels, we study commutative \match Rota-Baxter algebras and construct the free objects making use of the shuffle product with multiple decorations.
We also construct free commutative \match Rota-Baxter algebras in a relative context, to emulate the action of the integral operators on the coefficient functions in an integral equation. We finally show that free commutative \match Rota-Baxter algebras give the free \match Zinbiel algebra, generalizing the characterization of the shuffle product algebra as the free Zinbiel algebra obtained by Loday.
\end{abstract}

\subjclass[2010]{
13A99, 
16W99, 
16S10, 
08B20, 
45N05 
}

\keywords{matching Rota-Baxter algebra, matching dendriform algebra, shuffle product with decoration, matching Zinbiel algebra}
\maketitle

\tableofcontents

\setcounter{section}{0}

\allowdisplaybreaks

\section{Introduction}
In view of applications to Volterra integral equations with multiple kernels, this paper gives an explicit construction of free commutative \match Rota-Baxter algebras in both the absolute and relative contexts. The free commutative \match Rota-Baxter algebras also give the free object in the category of \match Zinbiel algebras.

\subsection{Integrations, Rota-Baxter algebras and shuffle products}
Algebraic approaches of integrals have led to the notions of the Rota-Baxter algebra and the shuffle product.

The notion of a Rota-Baxter algebra has its origin in the probability study of G. Baxter on fluctuation theory in probability~\mcite{Bax}. It is a pair $(R,P)$ where $R$ is an associative algebra and $P$ is a linear operator on $R$ satisfying the {\bf Rota-Baxter identity}
\begin{equation}
P(x)P(y)=P(xP(y))+P(P(x)y)+\lambda P(xy) \qforall x, y\in R.
\notag \mlabel{eq:rbo}
\end{equation}
Here $\lambda$ is a fixed scalar in the base ring. Even though the primary interest of Baxter was when $\lambda=-1$, he observed that the equation when $\lambda=0$ is the integration by parts formula for the {\bf simple Riemann integral} (here simple means having a trivial kernel)
\begin{equation}
I(f)(x):=\int_0^x f(t)\,dt
\mlabel{eq:riem}
\end{equation}
on continuous functions. Thus when $\lambda=0$, the pair $(R,P)$ can be called an integral algebra as an integral analog of the differential algebra~\mcite{Ko,PS, Ri} originated from the algebraic study of differential equations. The free differential algebra, realized as the differential polynomial algebra, provides a uniform setting to consider all differential equations, just like what the polynomial algebra provides for algebraic equations.
Indeed a large part of differential algebra is centered on developing an algebraic theory of differential equations in parallel to the theories of algebraic equations, including the corresponding Galois theory, algebraic groups and algorithms.

The recent decades witnessed a tremendous development of the Rota-Baxter algebra through its applications and connections in diverse areas in mathematics and physics~\mcite{Agu00,BBGN,BGN0,CK,EGM,Gub,GK,Ro}. However, much still needs to be done about Rota-Baxter algebras of weight zero as an algebraic framework for integral equations. In this regard, free commutative Rota-Baxter algebras of weight zero should serve as the universal space for integral equations with the simple integral operator in Eq.~\meqref{eq:riem}, similar to the role served by free differential algebras for differential equations. For some related literature, see~\mcite{BLRUV,GGR,RR}.

Incidently, shortly before Baxter introduced the algebraic formulation of the integration by parts formula that later bore his name, another notion also emerged in the algebraic study of integration, namely the one of the shuffle product. A fundamental notion in many areas from combinatorics to Hopf algebras and free Lie algebras~\mcite{EM,Ree,Reu}, the shuffle product has been the key in the study of integrals from various viewpoints. One is the theory of iterated path integrals invented by K.~T. Chen who, as a prominent differential geometer, pursued this subject throughout his career~\mcite{Ch1,Ch2,Ch3,CFL}.
The subject has since become an important tool in various branches of algebraic geometry, topology, number theory and mathematical physics~\mcite{Kr}.

In number theory, through the integral representations of multiple zeta values by Kontsevich, the shuffle product, together with the quasi-shuffle product, provides the algebraic study of multiple zeta values with the (extended) double shuffle framework that conjecturally dictates all the algebraic relations among multiple zeta values~\mcite{Br,GX,Ho1,Ho2,IKZ}.

On the categorical level, the ubiquitous role of the shuffle product algebra is manifested algebraically by the fact that the shuffle product algebra is the free Zinbiel algebra which is an equivalent form of the free commutative dendriform algebra~\mcite{Losh}.

Relating these two algebraic interpretations of the simple integral operator, the free commutative Rota-Baxter algebras of weight zero naturally builds on the shuffle product algebra, while commutative Rota-Baxter algebras in general give rise to Zinbiel algebras~\mcite{Agu00}.

\subsection{Multiple Rota-Baxter algebras and shuffle products with decorations}

With the simple integral operator in Eq.~\meqref{eq:riem} as the inceptive case, there are many other integral operators including the various Fredholm operators and Volterra operators~\mcite{Ze}. Further there can be different integral operators appearing in the same integral equation.
This presents the need to consider integral type or Rota-Baxter type algebras with multiple operators, leading to the subject of study of this paper.

Thus the purpose of this paper is twofold. On the one hand, we construct free commutative \match Rota-Baxter algebras as an algebraic framework to consider integral equations with multiple integral operators of Rota-Baxter nature. There we need to consider two cases. To begin with, the free commutative \match Rota-Baxter algebra is generated on an algebra. Further, since in an integral equation, the integral operators also act on the functions in the base algebra, there is the need to construct free \match Rota-Baxter algebras for which the generating algebra already has \match Rota-Baxter operators. The constructions are based on the shuffle product algebra from multiple decorations.

On the other hand, in analogous to Zinbiel algebras as the category in which to characterize the shuffle product algebra by a universal property~\mcite{Losh}, we introduce \match Zinbiel algebras in order to provide a suitable category for the universal property for the shuffle product algebra from multiple decorations.

With these goals in mind, here is the layout of the paper.
In Section~\mref{sec:freemrba}, we first recall notions and basic examples of \match Rota-Baxter algebras and the construction of free commutative Rota-Baxter algebra by the shuffle product. We then construct free commutative \match Rota-Baxter algebras by apply the shuffle product to an algebra with the multiple operators as decorations (Theorem~\mref{thm:comfree}).

In Section~\mref{sec:freerrba}, we introduce the notion of a relative Rota-Baxter algebra for the need of integral equations where the coefficient functions already carry an integral operator, as noted above. As in the case of commutative algebras, here relative means we only consider \match Rota-Baxter algebras on a given base \match Rota-Baxter algebra. We then construct free \match Rota-Baxter algebras in the relative context (Theorem~\mref{thm:freermrba}). The free object is a subalgebra of the free object in the non-relative case in Section~\mref{sec:freemrba}, but the linear operators have to be dealt with carefully in verifying that they satisfy the required conditions.

In Section~\mref{sec:mzinb}, we consider commutative dendriform algebras and its equivalent notion of Zinbiel algebras in the \match or multiple context. We then characterize the free commutative \match Rota-Baxter algebras from shuffle product with decorations in Section~\mref{sec:freemrba} as the free objects in the category of \match Zinbiel algebras (Theorem~\mref{thm:comu}), generalizing the shuffle product algebra construction of free Zinbiel algebras of Loday~\mcite{Losh}.

Other aspect of matching Rota-Baxter algebras were studied in~\mcite{GGZ,ZGG} with motivation from multiple pre-Lie algebras arising from the recent work~\mcite{BHZ,Foi18} on algebraic renormalization of regularity structures and polarized associative Yang-Baxter equations~\mcite{Agu00,Brz}. Another related study, on a more systematic algebraic approach to integral equations, can be found in~\mcite{GGL}. It is also worth noting that quite much progress has been made in another algebraic structure with multiple Rota-Baxter operators, called the Rota-Baxter family algebra with origin in a Lie theoretic approach to renormalization~\cite{EGP07,Gop}. This includes the constructions of free objects, related families of dendriform algebras and pre-Lie algebras, and the generalizations from the perspectives of monoidal categories and algebraic operad~\cite{Ag,Foi18,Foi20,FMZ,MZ,ZGM}.
\smallskip

\noindent {\bf Notation.}
Throughout this paper, let $\bfk$ be a unitary commutative ring
which will be the base ring of all modules, algebras,  tensor products, as well as linear maps, unless otherwise stated.
By an algebra we mean an associative unitary (\bfk-)algebra.

\section{Free commutative \match Rota-Baxter algebras}
\mlabel{sec:freemrba}
In this section, we first recall the concept of a matching Rota-Baxter algebra
and the construction of free commutative Rota-Baxter algebras utilizing the shuffle product.
We then apply the shuffle product to an algebra with multiple decorations in order to construct free commutative matching Rota-Baxter algebras.

\subsection{Matching Rota-Baxter algebras }
Let us recall the concept of  matching Rota-Baxter algebras~\mcite{ZGG} as a generalization of Rota-Baxter algebras.

\begin{defn}
Let $\Omega$ be a nonempty set and
$\lambda_\Omega:=(\lambda_\omega)_{\omega\in \Omega}$
a family of elements of $\bfk$ parameterized by $\Omega$. Equivalently, $\lambda_\Omega$ is a map $\lambda_\Omega: \Omega\to \bfk$. A {\bf matching \rba} (or simply {\bf MRBA}) of weight $\lambda_\Omega$ is a pair $(R, P_\Omega)$ consisting of an algebra $R$ and a family
$P_\Omega:=(P_\omega)_{\omega\in \Omega}$ of linear operators
$P_\omega: R\longrightarrow R, \omega\in \Omega\,$
that satisfy the {\bf matching Rota-Baxter equation}
\begin{align}
P_\alpha(x)P_{\beta}(y)&=P_{\alpha}\big(xP_{\beta}(y)\big) +P_{\beta}\big(P_{\alpha}(x)y\big)+\lambda_\beta P_{\alpha}(xy)
\, \text{ for }\, x,y \in R \text{ and } \alpha,\beta\in \Omega\,.
\notag
\end{align}
\end{defn}

For each $\omega\in \Omega$, $(R, P_\omega)$ is a Rota-Baxter algebra of weight $\lambda_\omega$. When $\lambda_\Omega=\{\lambda\}$,
that is, when the map $\lambda_\Omega: \Omega\to \bfk$ is a constant function, we also call the \mrba $(R,\lambda_\Omega)$ to have weight $\lambda$.
\begin{defn}
Let $(R,\, P_\Omega)$ and $(R',\,P'_\Omega)$ be \mrbas of the same weight $\lambda_\Omega$.
A linear map $f  : R\rightarrow R'$ is called an {\bf \mrba homomorphism} if $f $ is an algebra homomorphism  such that $f   P_\omega = P'_\omega f $ for all $\omega\in \Omega$.
\end{defn}
A natural example of Rota-Baxter operator of weight 0 is the operator of Riemann integral in Eq.~\meqref{eq:riem}. Consider the $\RR$-algebra $R:=\mathrm{Cont}(\RR)$ of continuous functions on $\RR$. Then $(R,I)$ is a Rota-Baxter algebra of weight 0~\mcite{Bax}.
We generalize this to the multiple case.
\begin{exam}
Fix a family $k_\omega(x)$ of functions (called kernels~\mcite{Ze}) in $R$ parameterized by $\omega\in \Omega$. Define the Volterra integral operators
\begin{equation} I_\omega: R\longrightarrow R, \quad
f(x)\mapsto \int_0^x k_\omega(t)f(t)\,dt, \quad \omega \in \Omega.
\mlabel{eq:mint}
\end{equation}

Note that $I_\omega(f)=I(k_\omega f)$ for the integral operator $I$ in Eq.~\eqref{eq:riem}.
Then from the Rota-Baxter property of $I$, we obtain, for $\alpha, \beta\in \Omega$ and $f, g\in R$,
$$
I_\alpha(f)I_\beta(g)
= I(k_\alpha f)I(k_\beta g) = I(k_\alpha f I(k_\beta g))+I(I(k_\alpha f) k_\beta g) = I_\alpha(f I_\beta(g)) + I_\beta (I_\alpha (f) g).
$$
Thus $(R,(I_\omega)_{\omega\in \Omega})$ is a \match \rba of weight zero. Understanding integral equations with such a family of Volterra operators is our main motivation in the construction of free commutative \mrbas, especially in the relative context (Section~\mref{sec:freerrba}). See~\mcite{GGL} for further study of Volterra operators and integral equations from an algebraic point of view.
\mlabel{ex:Zin3}
\end{exam}

Here are some further examples and properties of \mrbas~\mcite{ZGG}.
\begin{remark}
\begin{enumerate}
\item
Let $(R,P_\Omega)$ be an \mrba of weight $\lambda_\Omega$. For any linear combination
\begin{align*}
P:=\sum_{\omega\in \Omega} a_\omega P_\omega, \quad a_{\omega} \in \bfk,
\end{align*}
with finite support, the pair $(R, P)$ is a Rota-Baxter algebra of weight $\sum_{\omega\in \Omega} a_{\omega}\lambda_\omega$.

\item  \Match Rota-Baxter algebras have a close connection with \match pre-Lie algebras introduced by Foissy~\mcite{Foi18}.
Let $(R, P_\Omega)$ be a \mrba of weight $\lambda_\Omega$.  Define
\begin{align*}
x \ast_\omega y:= P_{\omega}(x)y-yP_{\omega}(x)-\lambda_\omega yx\, \text{ for } x,y,z\in R, \omega\in \Omega.
\end{align*}
Then the pair $(R, (\ast_\omega)_{\omega\in \Omega})$ is a \match pre-Lie algebra.

\item
For $r, s\in R\ot R$, let
$$ r_{13} s_{12} - r_{12}s_{23}+r_{23}s_{13}=-\lambda s_{13}$$
be the {\bf polarized associative Yang-Baxter equation} of weight $\lambda$. Then a solution of this equation gives a matching Rota-Baxter operator of weight $\lambda$. See~\mcite{ZGG} for details.
\end{enumerate}
\mlabel{remk:MRB}
\end{remark}

\subsection{ Free commutative \match Rota-Baxter algebras on an algebra}
We first recall the notion of shuffle product algebras and its application to the construction of free commutative Rota-Baxter algebras~\mcite{Gub,GK}. Then this process is extended to the construction of free commutative \mrbas.

Let $V$ be a $\bfk$-module. Let $\Shu(V)$ be the shuffle product algebra on $V$. So the underlying module of $\Shu(V)$ is the same as the tensor algebra
$$T(V):=\bigoplus_{k\geq 0}V^{\otimes k}=\bfk \oplus V \oplus V^{\ot 2}\oplus \cdots$$
but the multiplication is the shuffle product $\shap$~\cite{EM,GK}.
The product $\shap$ can be defined recursively:
For $\fraka=a_1\ot \fraka'\in V^{\ot m}$ and $\frakb=b_1\ot \frakb'\in V^{\ot n}$ with $\fraka'\in V^{\ot (m-1)}$ and $\frakb'\in V^{\ot (n-1)}$, we have
\begin{equation}
\fraka \shap \frakb=a_1\otimes(\fraka'\shap \frakb)+b_1\otimes (\fraka\shap \frakb'),
\mlabel{eq:shrec}
\end{equation}
with the usual convention that, if $m=1$ or $n=1$, then $\fraka'$ or $\frakb'$ is the identity in $\bfk$.

Now let $A$ be a commutative algebra.
We define
\begin{equation}
\sha(A):= A\otimes\Shu(A)=\bigoplus_{k\geq 1}A^{\otimes k}
\end{equation}
to be the tensor product algebra whose product $\diamond$ is called  the {\bf augmented  shuffle product}. More precisely, $\diamond$ is defined by
$$(a_0\ot \fraka) \diamond (b_0\ot \frakb) := (a_0b_0) \ot (\fraka\ssha \frakb)\,\text{ for }\, a_0, b_0\in A, \fraka, \frakb\in \Shu(A).$$
Define
$$P: \sha(A) \to \sha(A), \quad \fraka \mapsto P(\fraka)= 1\ot \fraka. $$
Then by~\mcite{GK}, the triple $(\sha(A), \diamond, P)$ is the free commutative Rota-Baxter algebra of weight zero  on the commutative algebra $A$.

\begin{defn}
Let  $\Omega$ be a nonempty set   and $A$ a commutative algebra. A {\bf free commutative \mrba}
of weight zero on $A$ is a commutative \mrba $\frbf(A)$ of weight zero together
with an algebra homomorphism $j_{A}:A \rightarrow \frbf(A)$ such that for any
commutative \mrba $(R, (P_{\omega, \, R})_{\omega \in \Omega})$
of weight zero and any algebra homomorphism $f: A \rightarrow R,$ there is a unique \mrba homomorphism
$\bar{f}: \frbf(A) \rightarrow R$ such that $f=\bar{f}  j_{A}.$
\end{defn}

We now construct a free commutative \mrba by a shuffle product algebra with decorations in the sense that the base algebra is extended to carry decorations from the multiple Rota-Baxter operators.\footnote{We thank the referee for pointing out that the decorated shuffle product in a previous version of the paper could be realized as the shuffle product on the module $\bfk\Omega\ot A$.}

Consider $\bfk \Omega \ot A$ which is simply the free $A$-module $A\Omega$, but we will keep the former notation to better separate the roles played by $A$ and $\Omega$. Define the shuffle product algebra
\begin{align}\mlabel{eq:shaplus}
\Shu(\bfk\Omega\ot A):=\bfk \oplus (\bfk \Omega \ot A)  \oplus (\bfk \Omega \ot A)^{\ot 2}\oplus \cdots,
\end{align}
that we also refer to as the {\bf shuffle product algebra with decorations} to emphasize that the generating module has a set $\Omega$ of decorations.
In order to simplify the notations and to avoid confusion among the multiple functions of the tensor symbol $\otimes$, we will use the column notation $\twovec{\omega}{a}$ for a pure tensor $\omega\ot a$ in $\bfk \Omega \ot A$ and use the symbol $\tot$ to denote the tensor symbol in the tensor power $(\bfk \Omega \ot A)^{\ot k}$. Thus we denote

\begin{equation}
\twovec{\omega_1}{a_1} \tot \twovec{\omega_2}{a_2}\tot \cdots \tot\twovec{\omega_k}{a_k}:=(\omega_1\ot a_1)\ot  \cdots \ot (\omega_k \ot a_k)\in (\bfk \Omega \ot A)^{\otimes k} \text{ for } k\geq 1.
\notag\mlabel{newds1}
\end{equation}

Further define
\begin{equation}\mlabel{eq:pureaug2b}
\msha(A):= A\otimes\Shu(\bfk\Omega\ot A)=A \bigoplus \left(\bigoplus_{k\geq 1}  \left( A \ot (\bfk \Omega \ot A)^{\otimes k}\right)\right)
\end{equation}
to be the tensor product algebra whose product is denoted by $\diamond$.
Note the difference of this notion with $\sha(\bfk\Omega \ot A):=(\bfk\Omega\ot A)\ot \Shu(\bfk\Omega\ot A)$ in the construction of free Rota-Baxter algebras of weight zero.

A pure tensor $\fraka
\in \msha(A)$ is of the form
\begin{align} \mlabel{eq:nota1}
\fraka =a_0\ot \frak{\bar a}= a_0\ot \twovec{\alpha_1}{a_1} \tot \bar{\fraka}',
\end{align}
where $a_0\in A$, $\frak{\bar a}\in A_\Omega^{\ot m}$ which is expressed as $\twovec{\alpha_1}{a_1} \tot \bar{\fraka}'$ for $a_1\in A, \alpha_1\in \Omega$ and $\bar{\fraka}'\in A_\Omega^{\ot (m-1)}, m\geq 1$ with the same convention as for the shuffle product in Eq.~\meqref{eq:shrec} that when $m=1$, $\bar{\fraka}'$ is the identity in $\bfk$.
Then together with another pure tensor
\begin{align}
\frakb:=\ b_0\ot \frak{\bar b} :=\ b_0\ot \twovec{\beta_1}{b_1}\tot \bar{\frakb}' := b_0\ot \twovec{\beta_1}{b_1}\tot \twovec{\beta_2} {b_2} \tot\cdots  \tot  \twovec{\beta_n} {b_n} \in A\ot (\bfk \Omega \ot A)^{\ot n} \subseteq \msha(A)
\label{eq:nota2}
\end{align}
with $n\geq 0$, we have
\begin{align}
\fraka \sham \frakb
= a_0b_0 \ot\left( \frak{\bar a} \shap  \frak{\bar b} \right)=a_0b_0 \ot\left(\,
\twovec{\alpha_1}{a_1} \tot\cdots\tot
\twovec{\alpha_m}{a_m}
\shap
\twovec{\beta_1}{b_1} \tot\cdots\tot \twovec{\beta_n}{b_n}
\right).
\mlabel{eq:matdia}
\end{align}

For $\omega\in \Omega$ and a pure tensor $\fraka=a_0\ot \bar{\fraka}\in \msha(A)$ expressed in the form of Eq.~\eqref{eq:nota1}, the expression $1 \ot \twovec{\omega}{a_0}
\tot \frak{\bar a}$ is well defined in $\msha(A)$. Thus we might define the map
\begin{equation}
P_\omega:= P_{\omega, \,A}: \msha(A) \rightarrow \msha(A),\quad \fraka=a_0\ot \frak{\bar a} \mapsto 1 \ot \twovec{\omega}{a_0}
\tot \frak{\bar a}.
\mlabel{eq:op}
\end{equation}

\begin{theorem}
Let  $\Omega$ be a \nes~  and $A$ a commutative algebra.
Then the triple $(\msha(A), \sham,\\\Po)$, together with the
natural embedding $j_{A}:A\rightarrow \msha(A)$, is the free commutative \mrba of weight zero on $A$.
\mlabel{thm:comfree}
\end{theorem}

\begin{proof}
We first show that $(\msha(A), \sham,  \Po)$ is a commutative \mrba of weight zero on $A$. By Eq.~\eqref{eq:matdia},  $(\msha(A), \sham)$ is a commutative algebra. So we only need to verify that $\Po$  satisfy the \match Rota-Baxter equation in the case of  $\lambda=0$.
For this we use Eq.~(\mref{eq:nota1}) and Eq.~(\mref{eq:nota2}) and check,
for $\alpha, \beta \in \Omega$ and pure tensors $\fraka, \frakb\in \sha_\Omega(A),$
we have
\begin{equation}
\begin{aligned}
& P_{\alpha }(\mathfrak{ a}) \sham P_{\beta }(\mathfrak{ b})\\
=&\ \left( 1 \ot \twovec{\alpha}{a_0}
\tot \frak{\bar a}\right)
\sham \left( 1 \ot \twovec{\beta}{b_0}
\tot \frak{\bar b}\right)  \quad (\text{by Eq.~(\mref{eq:op} }))\\
=&\ 1 \ot\Bigg(\left( \twovec{\alpha}{a_0}
\ot \frak{\bar a}\right)
\shap \left(  \twovec{\beta}{b_0}
\tot \frak{\bar b}\right) \Bigg) \quad (\text{by Eq.~(\mref{eq:matdia} }))\\
=&\ 1 \ot\twovec{\alpha}{a_0}\tot \left(
\frak{\bar a}
\shap \left(  \twovec{\beta}{b_0}
\tot \frak{\bar b}\right)\right)
+1 \ot \twovec{\beta}{b_0}\tot\left( \left(\twovec{\alpha}{a_0}
\tot \frak{\bar a}\right)
\shap \frak{\bar b} \right) \quad (\text{by Eq.~(\mref{eq:shrec} }))\\
=&\ P_{\alpha}\left({a_0}\ot \left(
\frak{\bar a}
\shap \left(  \twovec{\beta}{b_0}
\tot \frak{\bar b}\right)\right)\right)
+P_{\beta}\left({b_0}\ot\left( \left(\twovec{\alpha}{a_0}
\tot \frak{\bar a}\right)
\shap \frak{\bar b} \right)\right)  \quad (\text{by Eq.~(\mref{eq:op} }))\\
=&\ P_{\alpha}\left(({a_0}\ot \frak{\bar a}) \sham
 \left( 1\ot \twovec{\beta}{b_0}
\tot \frak{\bar b}\right)\right)
+P_{\beta}\left( \left(1\ot \twovec{\alpha}{a_0}
\tot \frak{\bar a}\right)
\sham (b_0\ot\frak{\bar b})\right)\quad (\text{by Eq.~(\mref{eq:matdia} }))\\
=&\  P_{\alpha}({\mathfrak{ a}} \sham P_{\beta }({b_0}
\ot \frak{\bar b}))
+P_{\beta}(P_{\alpha }({a_0}
\ot \frak{\bar a})\sham\mathfrak{b})\quad (\text{by Eq.~(\mref{eq:op} }))\\
=&\  P_{\alpha}({\mathfrak{ a}} \sham P_{\beta }(\mathfrak{ b}))
+P_{\beta}(P_{\alpha }({ \mathfrak{ a}})\sham \mathfrak{ b}).
\end{aligned}
\label{eq:cm}
\end{equation}

It remains to verify  the universal property of $(\msha(A),\, \Po)$. Let $(R, \, (P_{\omega,\,R })_{\omega\in \Omega})$ be a commutative \mrba  of weight zero and let
$f: A \rightarrow R$ be an algebra homomorphism. We construct a unique $\bar{f}: \msha(A) \to R$ with the desired universal property.

({\bf Existence}). To construct a linear map $\bar{f}: \msha(A) \rightarrow R$, it suffices to define $\bar{f}(\fraka)$ for pure tensor $\fraka=a_0\ot \bar{\fraka}\in A\ot (\bfk \Omega \ot A)^{\ot m}$,
with $m\geq 0$, $a_0\in A$ and $\bar{\fraka}=\twovec{\alpha_1}{a_1} \tot\bar{\fraka}'\in (\bfk \Omega \ot A)^{\ot m}$. For this we employ the induction on $m\geq 0$. For the initial step of $m=0$, we have $\fraka =a_0$
and define
\begin{equation} \mlabel{eq:barf0}
\bar{f}(a_0):=  f(a_0).
\end{equation}
For clarity, let us also treat the case of $m=1$ separately. Here $\fraka=a_0\ot \twovec{\alpha_1}{a_1}$ and we define
$$ \bar{f}(\fraka):= f(a_0)P_{\alpha,R}(a_1).$$
For the induction step for $m\geq 2$, we have
$\fraka=a_0\ot \twovec{\alpha_1}{a_1} \tot\bar{\fraka}'$.
We then define
\begin{equation}
\bar{f}(\mathfrak{a}) =\bar{f}\left(a_0\ot \twovec{\alpha_1}{a_1} \tot\bar{\fraka}'\right) := f(a_{0})P_{\alpha_1,\, R}(\bar{f}(a_1 \ot \bar{\fraka}')).
\mlabel{eq:fpp}
\end{equation}
Then for $\fraka\in \msha(A)$ and $\omega\in \Omega$,
$$(\bar{f} P_{\omega})(\fraka)=\bar{f}\left( P_{\omega}(\fraka)\right)
=\bar{f}(1 \ot \twovec{\omega}{a_0}\tot \frak{\bar a})
=f(1)P_{\omega ,\,R}(\bar{f}(a_0\ot \frak{\bar a}))=P_{\omega,\,R}(\bar{f}(\fraka))=(P_{\omega,\,R }  \bar{f})(\fraka).$$
Hence
\begin{equation}
\bar{f}  P_{\omega}=P_{\omega,\,R }  \bar{f} \, \text{ for }\, \omega\in \Omega.
\mlabel{eq:compo}
\end{equation}
We next prove the multiplicativity of $\bar{f}$:
\begin{equation}
 \bar{f}({\mathfrak{a}}\sham{\mathfrak{b}})=\bar{f}({\mathfrak{a}})\bar{f}({\mathfrak{b}})
\mlabel{eq:ff}
\end{equation}
for  $\mathfrak{a}=a_0\ot \twovec{\alpha_1}{a_1} \tot\bar{\fraka}' \in A\ot (\bfk \Omega \ot A)^{\ot m}$ and  $\mathfrak{b}=b_0\ot \twovec{\beta_1}{b_1} \tot\bar{\frakb}' \in A\ot (\bfk \Omega \ot A)^{\ot n}$,
by induction on $m+n\geq 0$. The initial step of $m=n=0$ follows from Eq.~\meqref{eq:barf0} and the multiplicativity of $f$.
Suppose that Eq.~(\mref{eq:ff}) has been verified for $m+n \leq k$ with $k\geq 0$, and consider the case of $m+n = k+1$.
Then we  have
\begin{align*}
&\ \bar{f}(\fraka\rsham \frakb)\\
=&\ \bar{f}\left(\left(a_0\ot \twovec{\alpha_1}{a_1}\tot \bar{\frak a}'\right) \rsham \left(b_0 \ot \twovec{\beta_1}{b_1}\tot \bar{\frak b}'\right)\right)\\
=&\ \bar{f}\left(a_0 b_0 \ot \left(\left( \twovec{\alpha_1}{a_1}\tot \bar{\frak a}'\right) \shap \left( \twovec{\beta_1}{b_1}\tot \bar{\frak b}'\right)\right)\right)\quad (\text{by Eq.~(\mref{eq:matdia})})\\
=&\ \bar{f}\left(a_0 b_0 \ot \twovec{\alpha_1}{a_1} \tot\left(\bar{\frak a}'\shap \left( \twovec{\beta_1}{b_1}\tot \bar{\frak b}'\right) \right)\right)+ \bar{f}\left(a_0 b_0 \ot \twovec{\beta_1}{b_1}\tot \left(\left( \twovec{\alpha_1}{a_1}\tot \bar{\frak a}'\right)\shap \bar{\frak b}'\right)\right)\quad (\text{by Eq.~(\mref{eq:shrec})})\\
=&\ f(a_0b_0) P_{\alpha_1, R}\bar{f}\left(a_1\ot \left(\bar{\fraka}'\shap \left( \twovec{\beta_1}{b_1}\tot \bar{\frak b}'\right)\right)\right)
+f(a_0b_0) P_{\beta_1, R}\bar{f}\left( b_1\ot \left( \left( \twovec{\alpha_1}{a_1}\tot \bar{\frak a}'\right) \shap \bar{\frakb}'\right)\right)\\
&\hspace{8cm}(\text{by Eq.~(\mref{eq:fpp})})\\
=&\ f(a_0b_0) P_{\alpha_1, R}\bar{f}\left((a_1\ot \bar{\fraka}')\rsham  \left( 1\ot \twovec{\beta_1}{b_1}\tot \bar{\frak b}'\right)\right)
+f(a_0b_0) P_{\beta_1, R}\bar{f}\left(  \left( 1\ot \twovec{\alpha_1}{a_1}\tot \bar{\frak a}'\right)\rsham (b_1\ot  \bar{\frakb}')\right)\\
&\hspace{8cm} (\text{by Eq.~(\mref{eq:matdia})})\\
=&\ f(a_0b_0) P_{\alpha_1, R}\bar{f}(a_1\ot \bar{\fraka}')\bar{f} \left( 1\ot \twovec{\beta_1}{b_1}\tot \bar{\frak b}'\right)
+f(a_0b_0) P_{\beta_1, R}\bar{f}  \left( 1\ot \twovec{\alpha_1}{a_1}\tot \bar{\frak a}'\right)\bar{f} (b_1\ot  \bar{\frakb}')\\
&\hspace{8cm}(\text{by the induction hypothesis})\\
=&\ f(a_0b_0) P_{\alpha_1, R}\bar{f}(a_1\ot \bar{\fraka}')\bar{f} ( P_{\beta_1}(b_1\ot \bar{\frak b}'))
+f(a_0b_0) P_{\beta_1, R}\bar{f}(P_{\alpha_1}({a_1}\ot \bar{\frak a}'))\bar{f} (b_1\ot  \bar{\frakb}')\quad (\text{by Eq.~(\mref{eq:op})})\\
=&\ f(a_0b_0) P_{\alpha_1, R}\Big(\bar{f}(a_1\ot \bar{\fraka}')P_{\beta_1, R}(\bar{f}(b_1\ot \bar{\frak b}'))\Big)
+f(a_0b_0) P_{\beta_1, R}\Big(P_{\alpha_1, R}(\bar{f}(a_1\ot \bar{\fraka}'))\bar{f}(b_1\ot \bar{\frak b}')\Big)\\
&\hspace{8cm} (\text{by Eq.~(\mref{eq:compo})})\\
=&\ f(a_0b_0) P_{\alpha_1, R}(\bar{f}(a_1\ot \bar{\fraka}'))P_{\beta_1, R}(\bar{f}(b_1\ot \bar{\frak b}'))\ \ (\text{by $(R, P_{\Omega, R})$ being a \match Rota-Baxter algebra})\\
=&\ f(a_0)f(b_0) P_{\alpha_1, R}(\bar{f}(a_1\ot \bar{\fraka}'))P_{\beta_1, R}(\bar{f}(b_1\ot \bar{\frak b}'))\quad (\text{by $f$ being an algebra homomorphism})\\
=&\ \Big(f(a_0) P_{\alpha_1, R}(\bar{f}(a_1\ot \bar{\fraka}'))\Big)\Big(f(b_0)P_{\beta_1, R}(\bar{f}(b_1\ot \bar{\frak b}'))\Big)\quad (\text{by the commutativity})\\
=&\ \bar{f}({\fraka})\bar{f}(\frakb)\quad (\text{by Eq.~(\mref{eq:dfu})}).
\end{align*}
This completes the induction.

\smallskip

\noindent
({\bf Uniqueness}). In fact, for ${\mathfrak{a}}=a_0\ot \twovec{\alpha_1}{a_1} \tot \twovec{\alpha_2}{a_2}\tot   \cdots  \tot \twovec{\alpha_m}{a_m}  \in A\ot (\bfk \Omega \ot A)^{\ot m} $, since
$$a_0\ot \twovec{\alpha_1}{a_1} \tot \twovec{\alpha_2}{a_2}\tot   \cdots  \tot \twovec{\alpha_m}{a_m}
=a_0P_{\alpha_{1}}\left(a_{1}P_{\alpha_{2}}(a_{2}\cdots P_{\alpha_{m}}(a_{m})\cdots)\right),$$
in order for $\bar{f}$ to be a \match Rota-Baxter algebra homomorphism, we must have
\begin{equation*}
\bar{f}(\mathfrak{a})=f(a_{0})P_{\alpha_1,\, R} \left(f(a_{1})P_{\alpha_2,\,R } (f(a_{2}) \cdots P_{\alpha_m,\,R}(f(a_m))\cdots )\right).
\mlabel{eq:fp}
\end{equation*}
So the uniqueness of $\bar{f}$ is proved.
This completes the proof.
\end{proof}

Let $X$ be a set and ${\bfk}[X]$ the polynomial algebra on $X$ with the natural embedding $X\hookrightarrow \bfk [X]$. As a special case of Theorem~\mref{thm:comfree}, we obtain a free commutative \match Rota-Baxter  algebra of weight zero on a set.

\begin{coro}
Let $X$ be a set.
The \match Rota-Baxter  algebra $(\shao(\bfk[X]), (P_{\omega, \,\bfk[X]})_{\omega \in \Omega})$ of weight zero, together with the natural embedding
$$j_{X}: X \hookrightarrow {\bfk}[X]  \hookrightarrow  \shao(\bfk[X]),$$
is the free commutative \match Rota-Baxter  algebra of weight zero on $X$,
described by the following universal property: for any  commutative \mrba $(R, \, (P_{\omega,\, R})_{\omega\in \Omega})$ of weight zero and any set map $f: X \rightarrow R$,
there exists a unique  \match \mrba homomorphism $\free{f}: \shao(\bfk [X]) \rightarrow R$ such that
$f=\bar{f} j_X$.
\end{coro}

\section{Free commutative relative matching Rota-Baxter algebras}
\mlabel{sec:freerrba}
With the application to integral equations~\mcite{GGL} in mind, we consider MRBAs in a relative context,
in the sense that the base ring is already an \mrba.

\subsection{Relative \mrbas and the construction of their free objects}
We begin with the following analog of algebras over a base ring.

\begin{defn}
Let $(\basf,\kappa_\Omega)$ be a fixed \mrba of weight $\lambda_\Omega$.
\begin{enumerate}
\item An {\bf $(\basf,\kappa_\Omega)$-\mrba} or simply an {\bf $(\basf,\kappa_\Omega)$-algebra} or, in general term, a relative \mrba, is an \mrba $(R,P_{\Omega, R})$ of weight $\lambda_\Omega$ together with an \mrba homomorphism $i=i_R:(\basf,\kappa_\Omega)\to (R,P_\Omega)$.

\item A homomorphism $f :(R,\, P_{\Omega,R})\to (R',\,P'_{\Omega,R'})$ of   $(\basf,\kappa_\Omega)$-algebras is
an $F$-algebra homomorphism $f  : R\rightarrow R'$ such that $$f  P_{\omega,\,R} = P'_{\omega,\,R'} f \,  \text{ for }\,  \omega\in \Omega.$$
\end{enumerate}
\end{defn}
As in the case of an $\basf$-algebra, the structure map $i_R$ is usually suppressed.

For an $(\basf,\kappa_\Omega)$-algebra $(R,P_\Omega)$, $R$ is an $\basf$-algebra and the following identity holds as a consequence of the \mrba homomorphism $i_R:(\basf,\kappa_\Omega)\to (R,P_\Omega)$.
\begin{align}
\kappa_\alpha(k)P_{\beta, R}(u) = P_{\alpha, R}\big( kP_{\beta, R}(u) \big)  + P_{\beta, R}\big(\kappa_\alpha(k)u\big)+ \lambda_{\beta} P_{\alpha, R}(k u) \, \text{ for }\, k\in \basf, u\in R, \alpha, \beta \in \Omega.
\label{eq:mrbm}
\end{align}

We note the presence to two base algebras $\bfk$ and $\basf$, where $\bfk$ is the ring of constants for linear maps and tensor products, while the Rota-Baxter operators $\kappa_\omega$ and $P_\omega, \omega \in \Omega,$ are not $\basf$-linear, as seen in the above equation.

The notion of a free $(\basf,\kappa_\Omega)$-\mrba is explicitly defined as follows.

\begin{defn}
Let $(\basf,\kappa_\Omega)$ be a fixed commutative \mrba of weight $\lambda_\Omega$ and  $A$ a commutative algebra.
A {\bf free $(\basf,\kappa_\Omega)$-\mrba on $A$} is an $(\basf,\kappa_\Omega)$-\mrba $(\frbf(A), P_{\Omega, A})$, together with an algebra homomorphism $j_{A}:A \rightarrow \frbf(A)$ satisfying the following universal property:
for any $(\basf,\kappa_\Omega)$-\mrba $(R, P_{\Omega, R})$ and any algebra homomorphism $f: A \rightarrow R,$
there is a unique $(\basf,\kappa_\Omega)$-\mrba homomorphism
$\bar{f}: \frbf(\basf, \kappa_\Omega) \rightarrow R$ such that $f=\bar{f} j_{A}.$
\end{defn}
When all the algebras involved are assumed to be commutative, we have the notion of a free commutative $(\basf,\kappa_\Omega)$-algebra on $A$.

Let $A$ be an augmented algebra. So $A$ can be taken as the unitization $A=\bfk\oplus A^+$ of an algebra $A^+$.
We will construct the free commutative relative \mrba of weight zero
on $A$. Let $(\basf,\kappa_\Omega)$ be a fixed
commutative \mrba of weight zero. Denote by
$$ \frakA:=\basf \ot A\,\text{ and }\, \frakA^+:= \basf \ot A^+.$$
So $\frakA=\basf \oplus \frakA^+$.

Recall the free commutative \mrba $\sha_\Omega(\frakA)$
of weight zero on $\frakA$ introduced in Eq.~\meqref{eq:pureaug2b}:
\begin{equation}
\msha(\frakA)= \frakA\ot \Shu(\bfk \Omega\ot \frakA)
=\bigoplus_{k\geq 0} \frakA \ot (\bfk\Omega\ot \frakA)^{\ot k}.
\notag\mlabel{eq:fmrbafa}
\end{equation}

Consider the shuffle product algebra
$\Shu(\bfk\Omega\ot \frakA^+)$ which is naturally a subalgebra of the shuffle product algebra with decorations $\Shu(\bfk\Omega\ot \frakA)$. Then we obtain the subalgebra
\begin{equation}
\mrsha(A):= \frakA\ot \Shu(\bfk\Omega \ot \frakA^+)=\frakA\ot \left(\bigoplus_{k\geq 0} (\bfk\Omega\ot \frakA^+)^{\ot k}\right)
=\frakA\oplus \left(\bigoplus_{k\geq 1} \frakA \ot (\bfk\Omega \ot\frakA^+)^{\ot k}\right)
\mlabel{eq:relsha2}
\end{equation}
of the free \mrba $\msha(\frakA)$ on $\frakA$.
Note that
$\mrsha(A)$ is spanned by pure tensors of the form
$$u_0\ot ({\omega_1} \ot u_1)\ot  \cdots \ot ({\omega_k}\ot u_k)=u_0\ot \twovec{\omega_1}{u_1}\tot \cdots \tot \twovec{\omega_k}{u_k}, \, u_0\in \frakA, u_1,\dots,u_k\in  \frakA^+, \omega_1, \dots, \omega_k\in \Omega.$$

Next we define linear operators
$$P_\omega: \mrsha(A) \rightarrow \mrsha(A)\,\text{ for }\, \omega\in \Omega.$$
Let $\fraku=u_0\ot \twovec{\omega_1}{u_1}\tot \cdots \tot \twovec{\omega_n}{u_n}$ be a pure tensor in $\mrsha(A)$.
Since $\frakA=\basf \oplus \frakA^+$, the first tensor factor $u_0$ is either in $\basf$ or in $\frakA^+$.
We accordingly define
\begin{equation}
P_{\omega}(\fraku):=P_{\omega,\basf,A}(\fraku):=\left\{\begin{array}{ll} \kappa_{\omega}(u_0), & \text{ if } u_0\in \basf, n=0, \\
\smallskip
\kappa_{\omega}(u_0)\ot \twovec{\omega_1} {u_1} \tot \twovec{\omega_2}{u_2}\tot \cdots \tot \twovec{\omega_n} {u_n} \\
\smallskip
\quad \ \,  -1\ot\twovec{\omega_1} {\kappa_{\omega}(u_0)u_1} \tot \twovec{\omega_2} {u_2} \tot\cdots \tot \twovec{\omega_n} {u_n},
& \text{ if } u_0\in \basf, n\geq 1, \\
\smallskip
1\ot \twovec{\omega} {u_0}\tot \twovec{\omega_1} {u_1} \tot \cdots \tot \twovec{\omega_n} {u_n}, & \text{ if } u_0\in \frakA^+. \end{array}\right.
\mlabel{eq:relrbo}
\end{equation}
Then $\mrsha(A)$ is closed under the operators $P_\omega$ for $\omega\in \Omega$. Denote by $P_{\Omega, \basf, A} := (P_{\omega})_{\omega\in \Omega}$.

Let
$$i: \basf \hookrightarrow \frakA \subseteq \mrsha(A)$$
and
$$j_A: A\hookrightarrow \frakA  \hookrightarrow \mrsha(A)$$
be the nature embeddings.

\begin{theorem}
Let $(\basf,\kappa_\Omega)$ be a commutative \mrba of weight zero and
$A$ a commutative augmented algebra.
Then the triple  $(\mrsha(A), \, \rsham, \, P_{\Omega, \basf, A})$,
together with the maps $i$ and $j_A$, is the free commutative $(\basf,\kappa_\Omega)$-\mrba on $A$.
\mlabel{thm:freermrba}
\end{theorem}

\subsection{The proof Theorem~\mref{thm:freermrba}}
We will carry out the proof of Theorem~\mref{thm:freermrba} into two steps:
\begin{enumerate}
\item[{\bf Step 1.}]
Together with $i:\basf \to \mrsha(A)$, the triple $(\mrsha(A), \, \rsham, \, P_{\Omega, \basf, A})$ is an $(\basf,\kappa_\Omega)$-\mrba;
\item[{\bf Step 2.}]
Together with $j_A: A\to \mrsha(A)$, the $(\basf,\kappa_\Omega)$-algebra $(\mrsha(A), \, \rsham, \, P_{\Omega, \basf, A})$ satisfies the universal property of a free commutative $(\basf,\kappa_\Omega)$-\mrba on $A$.
\end{enumerate}

\subsubsection{Step 1. The triple $(\mrsha(A), \, \rsham, \, P_{\Omega, \basf, A})$ is an $(\basf,\kappa_\Omega)$-\mrba}
Since $\rsham$ is commutative associative and, by Eq.~(\mref{eq:relrbo}), the map $i:\basf\to \mrsha(A)$ is compatible with the operators $\kappa_\omega$ on $\basf$ and $P_\omega, \omega\in \Omega$ on $\mrsha(A)$, it is sufficient to verify the identity
\begin{align*}
P_\omega(\fraku) \rsham P_{\bo} (\frakv)=P_\omega\Big(\fraku\rsham P_{\bo}(\frakv)\Big)+P_{\bo}\Big(P_\omega(\fraku)\rsham \frakv\Big) \, \text{ for }\, \fraku, \frakv\in \mrsha(A), \omega, \bo\in \Omega.
\end{align*}
By additivity, we only need to consider pure tensors $\fraku\in \frakA \ot (\bfk\Omega \ot\frakA^+)^{\ot m}$ and $\frakv\in \frakA \ot (\bfk\Omega \ot\frakA^+)^{\ot n}$.
Writing
\begin{align}
\fraku=&\ u_0\ot \twovec{\omega_1}{u_1}\tot \twovec{\omega_2}{u_2} \tot\cdots \tot \twovec{\omega_m} {u_m}=: u_0 \ot \bar{\fraku} :=u_0 \ot \twovec{\omega_1}{u_1}\tot {\bar{\frak u}'},\\
 \frakv=&\ v_0\ot\twovec{\bo_1}{v_1}\tot \twovec{\bo_2}{v_2} \tot\cdots \tot \twovec{\bo_n} {v_n}=: v_0 \ot \bar{\frakv}:= v_0 \ot \twovec{\bo_1}{v_1}\tot \bar{\frak v}',
\label{eq:nots}
\end{align}
we have four cases to consider.

\noindent{\bf Case 1.} $u_0\in \basf$ and $v_0\in \basf$. In this case, by Eq.~(\mref{eq:relrbo}), we have
\begin{align*}
P_\omega(\fraku)= \kappa_\omega(u_0)\ot \bar{\fraku}-1\ot \twovec{\omega_1} {\kappa_{\omega}(u_0)u_1} \tot\bar{\fraku}'\, \text{ and }\, P_{\bo} (\frakv)= \kappa_{\bo} (v_0)\ot \bar{\frakv}-1\ot \twovec{\bo_1} {\kappa_{\bo}(v_0)v_1} \tot\bar{\frakv}'.
\end{align*}
On the one hand,
\begin{align*}
&\ P_\omega\Big(u\rsham P_{\bo}(v)\Big)\\
=&\ P_\omega\left((u_0 \ot \bar{\fraku})\rsham \left(\kappa_{\bo} (v_0)\ot \bar{\frakv}-1\ot \twovec{\bo_1} {\kappa_{\bo}(v_0)v_1} \tot\bar{\frakv}'\right)\right)\\
=&\ P_\omega\left(u_0 \kappa_{\bo} (v_0)\ot (\bar{\fraku}\shap \bar{\frakv})\right)-P_\omega\left(u_0 \ot \left(\bar{\fraku} \shap  \left(\twovec{\bo_1} {\kappa_{\bo}(v_0)v_1} \tot\bar{\frakv}'\right)\right)\right)\quad(\text{by Eq.~(\mref{eq:matdia})})\\
=&\ P_\omega\left(u_0\kappa_{\bo}(v_0)\ot \twovec{\omega_1}{u_1}
\tot (\bar{\fraku}' \shap  \bar{\frakv})
+u_0\kappa_{\bo}(v_0) \ot  \twovec{\bo_1}{v_1} \tot(\bar{\fraku}\shap \bar{\frakv}') \right)\\
&\ -P_\omega\left(u_0\ot \twovec{\omega_1}{u_1}\tot \left({\bar{\fraku}' \shap \left( \twovec{\bo_1} {\kappa_{\bo}(v_0)v_1}\tot \bar{\frakv}'\right)}\right)
+u_0\ot\twovec {\bo_1}{\kappa_{\bo}(v_0)v_1}\tot ( \bar{\fraku}\shap \bar{\frakv}')\right)\\
&\hspace{8cm}(\text{by Eq.~(\mref{eq:shrec})})\\
=&\ \kappa_{\omega}(u_0\kappa_{\bo}(v_0))\ot\twovec{\omega_1}{u_1} \tot(\bar{\fraku}'\shap \bar{\frakv})
-1\ot \twovec{\omega_1}{\kappa_{\omega}(u_0\kappa_{\bo}(v_0))u_1} \tot(\bar{\fraku}' \shap \bar{\frakv})\\
&\ +\kappa_{\omega}(u_0\kappa_{\bo}(v_0))\ot\twovec{\bo_1}{v_1} \tot(\bar{\fraku}\shap \bar{\frakv}')
-1\ot\twovec{\bo_1}{\kappa_{\omega}(u_0\kappa_{\bo}(v_0))v_1}\tot (\bar{\fraku}\shap \bar{\frakv}')\\
&\ -\kappa_{\omega}(u_0)\ot \twovec{\omega_1}{u_1}
\tot \left(\bar{\fraku}' \shap \left(\twovec{\bo_1} {\kappa_{\bo}(v_0)v_1}\tot \bar{\frakv}'\right)\right)
+1\ot\twovec{\omega_1}{\kappa_{\omega}(u_0)u_1}
\tot \left(\bar{\fraku}' \shap \left(\twovec{\bo_1} {\kappa_{\bo}(v_0)v_1}\tot \bar{\frakv}'\right)\right)\\
&\ -\kappa_{\omega}(u_0)\ot \twovec {\bo_1}{\kappa_{\bo}(v_0)v_1}\tot ( \bar{\fraku}\shap \bar{\frakv}')
+1\ot \twovec {\bo_1}{\kappa_{\omega}(u_0)\kappa_{\bo}(v_0)v_1}\tot ( \bar{\fraku}\shap \bar{\frakv}')\\
&\hspace{8cm}(\text{by Eq.~(\mref{eq:relrbo})}).
\end{align*}
By the same argument, we have
\begin{align*}
&\ P_{\bo}\Big(P_\omega(u)\rsham v\Big)\\
=&\ P_{\bo}\left(\left(\kappa_\omega(u_0)\ot \bar{\fraku}-1\ot \twovec{\omega_1} {\kappa_{\omega}(u_0)u_1} \tot\bar{\fraku}'\right)\rsham (v_0 \ot\bar{\frakv})\right)\\
=&\ \kappa_{\bo}(\kappa_\omega(u_0)v_0)\ot\twovec{\omega_1}{u_1}\tot (\bar{\fraku}'\shap \bar{\frakv})
-1\ot\twovec{\omega_1}{\kappa_{\bo}(\kappa_\omega(u_0)v_0)u_1}\tot (\bar{\fraku}'\shap \bar{\frakv})\\
&\ +\kappa_{\bo}(\kappa_\omega(u_0)v_0)\ot\twovec{\bo_1}{v_1}\tot (\bar{\fraku}\shap \bar{\frakv}')
-1\ot\twovec{\bo_1}{\kappa_{\bo}(\kappa_\omega(u_0)v_0)v_1}\tot (\bar{\fraku}\shap \bar{\frakv}')\\
&\ -\kappa_{\bo}(v_0)\ot\twovec{\omega_1}{\kappa_\omega(u_0)u_1}\tot (\bar{\fraku}'\shap \bar{\frakv})
+1\ot\twovec{\omega_1}{\kappa_{\bo}(v_0)\kappa_\omega(u_0)u_1}\tot (\bar{\fraku}'\shap \bar{\frakv})\\
&\ -\kappa_{\bo}(v_0)\ot\twovec{\bo_1}{v_1}\tot\left(\left(\twovec{\omega_1}{\kappa_\omega(u_0)u_1}\tot \bar{\fraku}'\right)\shap \bar{\frakv}'\right)
+1\ot\twovec{\bo_1}{\kappa_{\bo}(v_0)v_1}\tot\left(\left(\twovec{\omega_1}{\kappa_\omega(u_0)u_1}\tot \bar{\fraku}'\right)\shap \bar{\frakv}'\right).
\end{align*}
Thus
\begin{equation}
\begin{aligned}
&\ P_\omega\Big(u\rsham P_{\bo}(v)\Big)+P_{\bo}\Big(P_\omega(u)\rsham v\Big)\\
=&\ \kappa_{\omega}(u_0\kappa_{\bo}(v_0))\ot\twovec{\omega_1}{u_1} \tot(\bar{\fraku}'\shap \bar{\frakv})
-1\ot \twovec{\omega_1}{\kappa_{\omega}(u_0\kappa_{\bo}(v_0))u_1} \tot(\bar{\fraku}' \shap \bar{\frakv})\\
&\ +\kappa_{\omega}(u_0\kappa_{\bo}(v_0))\ot\twovec{\bo_1}{v_1} \tot(\bar{\fraku}\shap \bar{\frakv}')
-1\ot\twovec{\bo_1}{\kappa_{\omega}(u_0\kappa_{\bo}(v_0))v_1}\tot (\bar{\fraku}\shap \bar{\frakv}')\\
&\ -\kappa_{\omega}(u_0)\ot \twovec{\omega_1}{u_1}
\tot \left(\bar{\fraku}' \shap \left(\twovec{\bo_1} {\kappa_{\bo}(v_0)v_1}\tot \bar{\frakv}'\right)\right)
+1\ot\twovec{\omega_1}{\kappa_{\omega}(u_0)u_1}
\tot \left(\bar{\fraku}' \shap \left(\twovec{\bo_1} {\kappa_{\bo}(v_0)v_1}\tot \bar{\frakv}'\right)\right)\\
&\ -\kappa_{\omega}(u_0)\ot \twovec {\bo_1}{\kappa_{\bo}(v_0)v_1}\tot ( \bar{\fraku}\shap \bar{\frakv}')
+1\ot \twovec {\bo_1}{\kappa_{\omega}(u_0)\kappa_{\bo}(v_0)v_1}\tot ( \bar{\fraku}\shap \bar{\frakv}')\\
&\ +
\kappa_{\bo}(\kappa_\omega(u_0)v_0)\ot\twovec{\omega_1}{u_1}\tot (\bar{\fraku}'\shap \bar{\frakv})
-1\ot\twovec{\omega_1}{\kappa_{\bo}(\kappa_\omega(u_0)v_0)u_1}\tot (\bar{\fraku}'\shap \bar{\frakv})\\
&\ +\kappa_{\bo}(\kappa_\omega(u_0)v_0)\ot\twovec{\bo_1}{v_1}\tot (\bar{\fraku}\shap \bar{\frakv}')
-1\ot\twovec{\bo_1}{\kappa_{\bo}(\kappa_\omega(u_0)v_0)v_1}\tot (\bar{\fraku}\shap \bar{\frakv}')\\
&\ -\kappa_{\bo}(v_0)\ot\twovec{\omega_1}{\kappa_\omega(u_0)u_1}\tot (\bar{\fraku}'\shap \bar{\frakv})
+1\ot\twovec{\omega_1}{\kappa_{\bo}(v_0)\kappa_\omega(u_0)u_1}\tot (\bar{\fraku}'\shap \bar{\frakv})\\
&\ -\kappa_{\bo}(v_0)\ot\twovec{\bo_1}{v_1}\tot\left(\left(\twovec{\omega_1}{\kappa_\omega(u_0)u_1}\tot \bar{\fraku}'\right)\shap \bar{\frakv}'\right)
+1\ot\twovec{\bo_1}{\kappa_{\bo}(v_0)v_1}\tot\left(\left(\twovec{\omega_1}{\kappa_\omega(u_0)u_1}\tot \bar{\fraku}'\right)\shap \bar{\frakv}'\right).
\end{aligned}\label{eq:rhs}
\end{equation}
Since $(\basf,\kappa_\Omega)$ is a commutative \mrba of weight zero, we have
\begin{align}
\kappa_{\bo}(v_0)\kappa_\omega(u_0)=\kappa_\omega(u_0)\kappa_{\bo}(v_0)
=\kappa_\omega(u_0\kappa_{\bo}(v_0))+\kappa_{\bo}(\kappa_\omega(u_0)v_0). \mlabel{eq:mrbaa}
\end{align}
This allows us to rewrite Eq.~(\mref{eq:rhs}) as
\begin{align*}
&\ P_\omega\Big(u\rsham P_{\bo}(v)\Big)+P_{\bo}\Big(P_\omega(u)\rsham v\Big)\\
=&\ \kappa_{\omega}(u_0\kappa_{\bo}(v_0))\ot\twovec{\omega_1}{u_1} \tot(\bar{\fraku}'\shap \bar{\frakv})
 +\kappa_{\omega}(u_0\kappa_{\bo}(v_0))\ot\twovec{\bo_1}{v_1} \tot (\bar{\fraku}\shap \bar{\frakv}')\\
&\ -\kappa_{\omega}(u_0)\ot \twovec{\omega_1}{u_1}
\tot \left(\bar{\fraku}' \shap \left(\twovec{\bo_1} {\kappa_{\bo}(v_0)v_1}\tot \bar{\frakv}'\right)\right)
+1\ot\twovec{\omega_1}{\kappa_{\omega}(u_0)u_1}
\tot \left(\bar{\fraku}' \shap \left(\twovec{\bo_1} {\kappa_{\bo}(v_0)v_1}\tot \bar{\frakv}'\right)\right)\\
&\ -\kappa_{\omega}(u_0)\ot \twovec {\bo_1}{\kappa_{\bo}(v_0)v_1}\tot ( \bar{\fraku}\shap \bar{\frakv}')
 +
\kappa_{\bo}(\kappa_\omega(u_0)v_0)\ot\twovec{\omega_1}{u_1}\tot (\bar{\fraku}'\shap \bar{\frakv})\\
&\ +\kappa_{\bo}(\kappa_\omega(u_0)v_0)\ot\twovec{\bo_1}{v_1}\tot (\bar{\fraku}\shap \bar{\frakv}')
 -\kappa_{\bo}(v_0)\ot\twovec{\omega_1}{\kappa_\omega(u_0)u_1}\tot (\bar{\fraku}'\shap \bar{\frakv})\\
&\ -\kappa_{\bo}(v_0)\ot\twovec{\bo_1}{v_1}\tot\left(\left(\twovec{\omega_1}{\kappa_\omega(u_0)u_1}\tot \bar{\fraku}'\right)\shap \bar{\frakv}'\right)
+1\ot\twovec{\bo_1}{\kappa_{\bo}(v_0)v_1}\tot\left(\left(\twovec{\omega_1}{\kappa_\omega(u_0)u_1}\tot \bar{\fraku}'\right)\shap \bar{\frakv}'\right)
\end{align*}
by gathering the second, tenth, fourteenth terms, and the fourth, eighth, twelfth terms of Eq.~(\mref{eq:rhs}).

On the other hand,
\begin{align*}
&\ P_\omega(u) \rsham P_{\bo} (v)\\
=&\ \left(\kappa_\omega(u_0)\ot \bar{\fraku}-1\ot \twovec{\omega_1} {\kappa_{\omega}(u_0)u_1} \tot \bar{\fraku}'\right)\rsham
\left(\kappa_{\bo} (v_0)\ot \bar{\frakv}-1\ot \twovec{\bo_1} {\kappa_{\bo}(v_0)v_1} \tot \bar{\frakv}'\right)\\
=&\ \kappa_\omega(u_0)\kappa_{\bo}(v_0)\ot(\bar{\fraku} \shap  \bar{\frakv}) -\kappa_{\omega}(u_0)\ot \left({\bar{\fraku} \shap \left( \twovec{\bo_1} {\kappa_{\bo}(v_0)v_1}\tot \bar{\frakv}'\right)}\right)\\
&\ -\kappa_{\bo}(v_0)\ot \left(\left(\twovec{\omega_1}{\kappa_\omega(u_0)u_1}\tot \bar{\fraku}'\right)\shap \bar{\frakv}\right)
+1\ot \left(\left(\twovec{\omega_1}{\kappa_\omega(u_0)u_1}\tot \bar{\fraku}'\right)\shap
\left(\twovec{\bo_1} {\kappa_{\bo}(v_0)v_1} \tot \bar{\frakv}'\right)\right)\\
=&\ \kappa_\omega(u_0)\kappa_{\bo}(v_0)\ot \twovec{\omega_1}{u_1}
\tot (\bar{\fraku}' \shap  \bar{\frakv})
+\kappa_\omega(u_0)\kappa_{\bo}(v_0) \ot  \twovec{\bo_1}{v_1} \tot (\bar{\fraku}\shap \bar{\frakv}')\\
&\ -\kappa_{\omega}(u_0)\ot \twovec{\omega_1}{u_1}\tot \left({\bar{\fraku}' \shap \left( \twovec{\bo_1} {\kappa_{\bo}(v_0)v_1}\tot \bar{\frakv}'\right)}\right)
-\kappa_{\omega}(u_0)\ot \twovec{\bo_1}{\kappa_{\bo}(v_0)v_1}\tot(\bar{\fraku} \shap  \bar{\frakv}') \\
&\ -\kappa_{\bo}(v_0)\ot\twovec{\omega_1}{\kappa_\omega(u_0)u_1}\tot (\bar{\fraku}'\shap \bar{\frakv})
-\kappa_{\bo}(v_0)\ot\twovec{\bo_1}{v_1}\tot\left(\left(\twovec{\omega_1}{\kappa_\omega(u_0)u_1}\tot \bar{\fraku}'\right)\shap \bar{\frakv}'\right)\\
&\ +1\ot\twovec{\omega_1}{\kappa_{\omega}(u_0)u_1}
\tot \left(\bar{\fraku}' \shap \left(\twovec{\bo_1} {\kappa_{\bo}(v_0)v_1}\tot \bar{\frakv}'\right)\right)
+1\ot\twovec{\bo_1}{\kappa_{\bo}(v_0)v_1}\tot\left(\left(\twovec{\omega_1}{\kappa_\omega(u_0)u_1}\tot \bar{\fraku}'\right)\shap \bar{\frakv}'\right)\\
&\hspace{8cm}(\text{by Eq.~(\mref{eq:shrec})})\\
=&\ \kappa_\omega(u_0)\kappa_{\bo}(v_0)\ot \twovec{\omega_1}{u_1}
\tot (\bar{\fraku}' \shap  \bar{\frakv})
+\kappa_\omega(u_0)\kappa_{\bo}(v_0) \ot  \twovec{\bo_1}{v_1} \tot (\bar{\fraku}\shap \bar{\frakv}')\\
&\ -\kappa_{\omega}(u_0)\ot \twovec{\omega_1}{u_1}\tot \left({\bar{\fraku}' \shap \left( \twovec{\bo_1} {\kappa_{\bo}(v_0)v_1}\tot \bar{\frakv}'\right)}\right)
-\kappa_{\omega}(u_0)\ot \twovec{\bo_1}{\kappa_{\bo}(v_0)v_1}\tot(\bar{\fraku} \shap  \bar{\frakv}') \\
&\ -\kappa_{\bo}(v_0)\ot\twovec{\omega_1}{\kappa_\omega(u_0)u_1}\tot (\bar{\fraku}'\shap \bar{\frakv})
-\kappa_{\bo}(v_0)\ot\twovec{\bo_1}{v_1}\tot\left(\left(\twovec{\omega_1}{\kappa_\omega(u_0)u_1}\tot \bar{\fraku}'\right)\shap \bar{\frakv}'\right)\\
&\ +1\ot\twovec{\omega_1}{\kappa_{\omega}(u_0)u_1}
\tot \left(\bar{\fraku}' \shap \left(\twovec{\bo_1} {\kappa_{\bo}(v_0)v_1}\tot \bar{\frakv}'\right)\right)
+1\ot\twovec{\bo_1}{\kappa_{\bo}(v_0)v_1}\tot\left(\left(\twovec{\omega_1}{\kappa_\omega(u_0)u_1}\tot \bar{\fraku}'\right)\shap \bar{\frakv}'\right)\\
=&\ \kappa_\omega(u_0\kappa_{\bo}(v_0)) \ot \twovec{\omega_1}{u_1}
\tot (\bar{\fraku}' \shap  \bar{\frakv})
+\kappa_\omega(u_0\kappa_{\bo}(v_0)) \ot \twovec{\omega_1}{u_1}
\tot (\bar{\fraku}' \shap  \bar{\frakv})\\
&\ +\kappa_\omega(u_0\kappa_{\bo}(v_0)) \ot  \twovec{\bo_1}{v_1} \tot (\bar{\fraku}\shap \bar{\frakv}')
+\kappa_\omega(u_0\kappa_{\bo}(v_0)) \ot  \twovec{\bo_1}{v_1} \tot (\bar{\fraku}\shap \bar{\frakv}')\\
&\ -\kappa_{\omega}(u_0)\ot \twovec{\omega_1}{u_1}\tot \left({\bar{\fraku}' \shap \left( \twovec{\bo_1} {\kappa_{\bo}(v_0)v_1}\tot \bar{\frakv}'\right)}\right)
-\kappa_{\omega}(u_0)\ot \twovec{\bo_1}{\kappa_{\bo}(v_0)v_1}\tot(\bar{\fraku} \shap  \bar{\frakv}') \\
&\ -\kappa_{\bo}(v_0)\ot\twovec{\omega_1}{\kappa_\omega(u_0)u_1}\tot (\bar{\fraku}'\shap \bar{\frakv})
-\kappa_{\bo}(v_0)\ot\twovec{\bo_1}{v_1}\tot\left(\left(\twovec{\omega_1}{\kappa_\omega(u_0)u_1}\tot \bar{\fraku}'\right)\shap \bar{\frakv}'\right)\\
&\ +1\ot\twovec{\omega_1}{\kappa_{\omega}(u_0)u_1}
\tot \left(\bar{\fraku}' \shap \left(\twovec{\bo_1} {\kappa_{\bo}(v_0)v_1}\tot \bar{\frakv}'\right)\right)
+1\ot\twovec{\bo_1}{\kappa_{\bo}(v_0)v_1}\tot\left(\left(\twovec{\omega_1}{\kappa_\omega(u_0)u_1}\tot \bar{\fraku}'\right)\shap \bar{\frakv}'\right)\\
&\hspace{8cm}(\text{by Eq.~(\mref{eq:mrbaa})}).
\end{align*}
Then the $i$-th term in the expansion of $P_\omega(u) \rsham P_{\bo} (v)$ is equal to the $\sigma(i)$-th term in the
 expansion of $P_\omega\Big(u\rsham P_{\bo}(v)\Big)+P_{\bo}\Big(P_\omega(u)\rsham v\Big)$, where $\sigma$ is the permutation of order $10$:
 \begin{equation*}
\begin{pmatrix}
     i \\
     \sigma(i)
\end{pmatrix}
=
\left(
\begin{array}{ccccccccccc}
1 & 2 & 3 & 4 & 5 & 6 & 7 & 8 & 9 & 10\\
1 & 6 & 2 & 7 & 3 & 5 & 8 & 9 & 4 & 10\emph{}
\end{array}
\right).
\end{equation*}
\noindent{\bf Case 2.} $u_0\in \basf$ and $v_0\in \frakA^+$. In this case, by Eq.~(\mref{eq:relrbo}), we have
\begin{align*}
P_\omega(u)= \kappa_\omega(u_0)\ot \bar{\fraku}-1\ot \twovec{\omega_1} {\kappa_{\omega}(u_0)u_1} \tot \bar{\fraku}' \, \text{ and }\, P_{\bo} (v)= 1\ot \twovec{\bo}{v_0}\tot\bar{\frakv}.
\end{align*}
Then on the one hand,
\begin{align*}
&\ P_\omega(u) \rsham P_{\bo} (v)\\
=&\ \left(\kappa_\omega(u_0)\ot \bar{\fraku}-1\ot \twovec{\omega_1} {\kappa_{\omega}(u_0)u_1} \tot \bar{\fraku}'\right) \rsham \left(1\ot \twovec{\bo}{v_0}\tot \bar{\frakv}\right)\\
=&\ \kappa_\omega(u_0)\ot \left(\bar{\fraku}\shap \left(\twovec{\bo}{v_0}\tot \bar{\frakv}\right)\right)
-1\ot \left(\left(\twovec{\omega_1} {\kappa_{\omega}(u_0)u_1} \tot \bar{\fraku}'\right) \shap \left( \twovec{\bo}{v_0}\tot \bar{\frakv}\right)\right)\\
=&\ \kappa_\omega(u_0)\ot \twovec{\omega_1}{u_1}\tot \left(\bar{\fraku}'\shap \left(\twovec{\bo}{v_0}\tot \bar{\frakv}\right)\right)
+\kappa_\omega(u_0)\ot \twovec{\bo}{v_0}\tot \left(\bar{\fraku}\shap  \bar{\frakv}\right)\\
&\ -1\ot \twovec{\omega_1} {\kappa_{\omega}(u_0)u_1} \tot \left(\bar{\fraku}' \shap \left( \twovec{\bo}{v_0}\tot \bar{\frakv}\right)\right)
-1\ot \twovec{\bo}{v_0}\tot \left(\left(\twovec{\omega_1} {\kappa_{\omega}(u_0)u_1} \tot \bar{\fraku}'\right) \shap  \bar{\frakv}\right)\\
&\hspace{8cm} (\text{by Eq.~(\mref{eq:shrec})})\\
=&\ \kappa_\omega(u_0)\ot \twovec{\omega_1}{u_1}\tot \left(\bar{\fraku}'\shap \left(\twovec{\bo}{v_0}\tot \bar{\frakv}\right)\right)
+\kappa_\omega(u_0)\ot \twovec{\bo}{v_0}\tot \left(\bar{\fraku}\shap  \bar{\frakv}\right)\\
&\ -1\ot \twovec{\omega_1} {\kappa_{\omega}(u_0)u_1} \tot \left(\bar{\fraku}' \shap \left( \twovec{\bo}{v_0}\tot \bar{\frakv}\right)\right)
-1\ot \twovec{\bo}{v_0}\tot \twovec{\omega_1} {\kappa_{\omega}(u_0)u_1} \tot (\bar{\fraku}' \shap  \bar{\frakv})\\
&\ -1\ot\twovec{\bo}{v_0}\tot \twovec{\bo_1}{v_1}\tot \left(\left(\twovec{\omega_1} {\kappa_{\omega}(u_0)u_1} \tot \bar{\fraku}'\right) \shap  \bar{\frakv}'\right)\quad (\text{by Eq.~(\mref{eq:shrec})}).
\end{align*}
On the other hand,
\begin{align*}
&\ P_\omega\Big(u\rsham P_{\bo}(v)\Big)+P_{\bo}\Big(P_\omega(u)\rsham v\Big)\\
=&\ P_\omega \left((u_0 \ot \bar{\fraku})\rsham \left(1\ot \twovec{\bo}{v_0}\tot\bar{\frakv}\right)\right)
+P_{\bo} \left(\left(\kappa_\omega(u_0)\ot \bar{\fraku}-1\ot \twovec{\omega_1} {\kappa_{\omega}(u_0)u_1} \tot \bar{\fraku}'\right)\rsham (v_0\ot \bar{\frakv})\right)\\
=&\ P_\omega \left(u_0 \ot \left( \bar{\fraku}\shap \left(\twovec{\bo}{v_0}\tot\bar{\frakv}\right)\right)\right)
+P_{\bo} \left(\kappa_\omega(u_0)v_0\ot (\bar{\fraku} \shap  \bar{\frakv})-v_0 \ot \left(\twovec{\omega_1} {\kappa_{\omega}(u_0)u_1} \tot \bar{\fraku}'\right)\shap \bar{\frakv}\right)\\
&\hspace{9cm} (\text{by Eq.~(\mref{eq:matdia})})\\
=&\ P_\omega\left(u_0\ot \twovec{\omega_1}{u_1}\tot \left(\bar{\fraku}' \shap \left(\twovec{\bo}{v_0} \tot \bar{\frakv} \right)\right)+u_0\ot \twovec{\bo}{v_0}\tot (\bar{\fraku}\shap \bar{\frakv}) \right)\\
&\ +P_{\bo}\left(\kappa_{\omega}(u_0)v_0\ot\twovec{\omega_1}{u_1}\tot (\bar{\fraku}'\shap \bar{\frakv})
+\kappa_{\omega}(u_0)v_0\ot\twovec{\bo_1}{v_1} \tot (\bar{\fraku}\shap \bar{\frakv}')\right)\\
&\ -P_{\bo}\left(v_0 \ot\twovec{\omega_1}{\kappa_{\omega}(u_0)u_1} \tot(\bar{\fraku}'\shap \bar{\frakv})
+v_0\ot\twovec{\bo_1}{v_1} \tot \left( \left(\twovec{\omega_1}{\kappa_{\omega}(u_0)u_1} \tot \bar{\fraku}'\right)\shap  \bar{\frakv}'\right)\right)\quad (\text{by Eq.~(\mref{eq:matdia})})\\
=&\ \kappa_\omega (u_0)\ot \twovec{\omega_1}{u_1}\tot \left(\bar{\fraku}' \shap \left(\twovec{\bo}{v_0} \tot \bar{\frakv} \right)\right)
-1\ot \twovec{\omega_1}{\kappa_\omega (u_0)u_1}\tot \left(\bar{\fraku}' \shap \left(\twovec{\bo}{v_0} \tot \bar{\frakv} \right)\right)\\
&\ +\kappa_\omega (u_0)\ot  \twovec{\bo}{v_0}\tot (\bar{\fraku}\shap \bar{\frakv})
-1\ot  \twovec{\bo}{\kappa_\omega (u_0)v_0}\tot (\bar{\fraku}\shap \bar{\frakv})\\
&\ +1\ot \twovec{\bo}{\kappa_{\omega}(u_0)v_0}\tot\twovec{\omega_1}{u_1}\tot (\bar{\fraku}'\shap \bar{\frakv})
+1\ot \twovec{\bo}{\kappa_{\omega}(u_0)v_0}\tot\twovec{\bo_1}{v_1} \tot (\bar{\fraku}\shap \bar{\frakv}')\\
&\ -1\ot \twovec{\bo}{v_0} \tot \twovec{\omega_1}{\kappa_{\omega}(u_0)u_1} \tot(\bar{\fraku}'\shap \bar{\frakv})
-1\ot \twovec{\bo}{v_0}\tot \twovec{\bo_1}{v_1} \tot \left( \left(\twovec{\omega_1}{\kappa_{\omega}(u_0)u_1} \tot \bar{\fraku}'\right)\shap  \bar{\frakv}'\right)\\
&\hspace{9cm}(\text{by Eq.~(\mref{eq:relrbo})})\\
=&\ \kappa_\omega (u_0)\ot \twovec{\omega_1}{u_1}\tot \left(\bar{\fraku}' \shap \left(\twovec{\bo}{v_0} \tot \bar{\frakv} \right)\right)
-1\ot \twovec{\omega_1}{\kappa_\omega (u_0)u_1}\tot \left(\bar{\fraku}' \shap \left(\twovec{\bo}{v_0} \tot \bar{\frakv} \right)\right)\\
&\ +\kappa_\omega (u_0)\ot  \twovec{\bo}{v_0}\tot (\bar{\fraku}\shap \bar{\frakv})
-1\ot \twovec{\bo}{v_0} \tot \twovec{\omega_1}{\kappa_{\omega}(u_0)u_1} \tot(\bar{\fraku}'\shap \bar{\frakv})\\
&\ -1\ot \twovec{\bo}{v_0}\tot \twovec{\bo_1}{v_1} \tot \left( \left(\twovec{\omega_1}{\kappa_{\omega}(u_0)u_1} \tot \bar{\fraku}'\right)\shap  \bar{\frakv}'\right),
\end{align*}
where the last equality is by gathering the fourth, fifth, sixth terms by Eq.~(\mref{eq:shrec}).

Thus we obtain
\begin{align*}
P_\omega(\fraku) \rsham P_{\bo} (\frakv)=P_\omega\Big(\fraku\rsham P_{\bo}(\frakv)\Big)+P_{\bo}\Big(P_\omega(\fraku)\rsham \frakv\Big).
\end{align*}

\noindent{\bf Case 3.} $u_0\in \frakA^+$ and $v_0\in \basf$. The proof is similar to the one for Case 2.
\smallskip

\noindent{\bf Case 4.} $u_0\in \frakA^+$ and $v_0\in \frakA^+$. In this case, we can write
\begin{align*}
P_{\omega}(\fraku)=1\ot \twovec{\omega}{u_0}\tot\bar{\fraku}\, \text{ and }\, P_{\bo} (\frakv)= 1\ot \twovec{\bo}{v_0}\tot\bar{\frakv}.
\end{align*}
Then the proof in this case is similar to the proof of Eq.~(\mref{eq:cm}) in Theorem~\mref{thm:comfree}.

This completes the proof of Step 1.

\subsubsection{Step 2. The universal property of $(\mrsha(A), \, \rsham, \, P_{\Omega, \basf, A})$}

Let $(R, P_{\Omega, R})$
be a commutative $(\basf,\kappa_\Omega)$-\mrba of weight zero and let $f: A \rightarrow R$ be an algebra homomorphism. We show that there is unique $(\basf,\kappa_\Omega)$-\mrba homomorphism $\free{f}:(\mrsha(A),P_\Omega)\to (R,P_{\Omega,R})$ such that $\free{f} j_R=f$.

({\bf The existence}).
For any pure tensor $\fraku=u_0\ot \bar{u}
 \in \frakA\ot (\bfk \Omega \ot \frakA^+)^{\ot m}$ with $m\geq 0$, we use induction on $m\geq 0$ to define $\free{f}(u)$.
For the initial step of $m=0$, we have $\fraku = u_0 \ot 1_\bfk=u_0$.
If $u_0$ is in $\basf$, then define
\begin{align}
\bar{f}(u):= u. \mlabel{eq:ini}
\end{align}
If $u_0$ is in $\frakA^+$, then we may write $u_0=u_{0,a} \ot u_{0, v}=u_{0,a} u_{0, v}$ for some $u_{0, a}\in \basf$ and $u_{0,v}\in A^+$,
and we define
\begin{align}
\bar{f}(u_0) :=\bar{f}(u_{0,a} u_{0, v}):=u_{0,a} f(u_{0, v}). \mlabel{eq:ini2}
\end{align}

Consider the induction step of $m\geq 1$. Write $\fraku=u_0\ot \twovec{\omega_1}{u_1}\tot \twovec{\omega_2}{u_2} \tot\cdots \tot \twovec{\omega_m} {u_m}= u_0\ot \twovec{\omega_1}{u_1}\tot \bar{\fraku}'.$
If $u_0\in \basf$, then define
\begin{equation}
\bar{f}(\fraku) :=\bar{f}\left(u_0\ot \twovec{\omega_1}{u_1}\tot \bar{\fraku}'\right) =\bar{f}(u_0P_{\omega_1}(u_1\ot \bar{\fraku}'))
:= u_0\bar{f}(P_{\omega_1}(u_1\ot \bar{\fraku}'))  :=u_0P_{\omega_1, R}\bar{f}(u_1\ot \bar{\fraku}').
\mlabel{eq:dfu}
\end{equation}
If $u_0\in \frakA^+$, then define
\begin{align}
\bar{f}\left(u_0\ot \twovec{\omega_1}{u_1}\tot \bar{\fraku}'\right) :=\bar{f}({u_0} P_{\omega_1}(u_1\ot \bar{\fraku}')):=\bar{f}(u_0)P_{\omega_1, R}\bar{f}(u_1\ot \bar{\fraku}').
\mlabel{eq:dfuu2}
\end{align}

By Eqs.~(\mref{eq:ini}) and~(\mref{eq:dfu}), $\bar{f}$ is $\basf$-linear. Now we prove that $\bar{f}$ is compatible with the operators $P_\omega$ and $P_{\omega,R}, \omega \in \Omega$.

Let $\fraku=u_0\ot \twovec{\omega_1}{u_1}\tot \bar{\fraku}' \in \frakA\ot (\bfk \Omega \ot \frakA^+)^{\ot m}$ with $m\geq 0$ and $P_\omega\in P_\Omega$.
If $u_0\in \basf$, then
\begin{align*}
\bar{f} P_{\omega}(\fraku)
=&\ \bar{f}P_{\omega}\left(u_0\ot \twovec{\omega_1}{u_1}\tot \bar{\fraku}'\right)\\
=&\ \bar{f}\,\left(\kappa_\omega(u_0)\ot \twovec{\omega_1}{u_1}\tot \bar{\fraku}'-1\ot \twovec{\omega_1} {\kappa_{\omega}(u_0)u_1} \tot \bar{\fraku}'\right)
\quad (\text{by Eq.~(\mref{eq:relrbo})})\\
=&\ \kappa_\omega(u_0) P_{\omega_1, R}(\bar{f}(u_1\ot \bar{\fraku}'))- P_{\omega_1, R}\Big(\bar{f}(\kappa_\omega(u_0)u_1\ot \bar{\fraku}')\Big)\quad (\text{by Eq.~(\mref{eq:dfu})})\\
=&\ \kappa_\omega(u_0) P_{\omega_1, R}(\bar{f}(u_1\ot \bar{\fraku}'))-P_{\omega_1, R}\Big(\kappa_\omega(u_0)\bar{f}(u_1\ot \bar{\fraku}')\Big)\quad (\text{by $\basf$-linearity})\\
=&\ P_{\omega, R}\Big(u_0 P_{\omega_1, R}(\bar{f}(u_1\ot \bar{\fraku}')) \Big)\quad (\text{by Eq.~\meqref{eq:mrbm}})\\
=&\ P_{\omega, R} \bar{f}\left(u_0\ot  \twovec{\omega_1}{u_1} \tot \bar{\fraku}'\right)\quad (\text{by Eq.~(\mref{eq:dfu})})\\
=&\ P_{\omega, R} \bar{f}(\fraku).
\end{align*}
If $u_0\in \frakA^+$, then
\begin{align*}
\bar{f} P_{\omega}(\fraku)=&\ \bar{f}\left(1\ot \twovec{\omega}{u_0} \tot \twovec{\omega_1}{u_1}\tot \bar{\fraku}'\right)=\bar{f}(1)P_{\omega, R}\bar{f} (\fraku)=P_{\omega, R}\bar{f} (\fraku).
\end{align*}
Thus
\begin{equation}
\bar{f}  P_{\omega}=P_{\omega, R}  \bar{f}\, \text{ for }\, \omega\in \Omega.
\mlabel{eq:cm22}
\end{equation}

Next we check the compatibility of $\bar{f}$ with the multiplication $\rsham$:
for $\fraku=u_0\ot \twovec{\omega_1}{u_1}\tot \bar{\fraku}' \in \frakA\ot (\bfk\Omega \ot\frakA^+)^{\ot m}$ and $\frakv=v_0 \ot \twovec{\bo_1}{v_1}\tot \bar{\frak v}' \in \frakA\ot (\bfk\Omega \ot\frakA^+)^{\ot n}$,
\begin{equation}
 \bar{f}({\fraku}\rsham{\frakv})=\bar{f}({\fraku})\bar{f}(\frakv).
\mlabel{eq:comf}
\end{equation}
We will check this utilizing the induction on $m+n\geq 0$. When $m=n=0$,  we have
$$\fraku= u_0\,\text{ and }\,  \frakv=v_0.$$
If $u_0\in \basf$ and $v_0\in \basf$, then it follows from Eq.~(\mref{eq:ini}) that
$$\bar{f}(\fraku\rsham \frakv)=\bar{f}( u_0\rsham  v_0)=\bar{f}( u_0v_0)= u_0v_0=\bar{f}( u_0)\bar{f}( v_0)=\bar{f}(\fraku)\bar{f}(\frakv).$$
If $u_0\in \basf$ and $v_0\in \frakA^+$ or $u_0\in \frakA^+$ and $v_0\in \basf$, without loss of generality,
letting $u_0\in \basf$ and $v_0\in \frakA^+$, then we may write $v_0=v_{0,a} \ot v_{0, v}=v_{0,a}  v_{0, v}$ for some $v_{0, a}\in \basf$ and $v_{0,v}\in A^+$,
and we have
$$\bar{f}(\fraku\rsham \frakv)=\bar{f}(u_0\rsham v_0)=\bar{f}(u_0 v_{0, a} v_{0,v})= u_0 \bar{f}(v_{0, a}v_{0,v})
=\bar{f}(u_0) \bar{f}(v_0) = \bar{f}(\fraku) \bar{f}(\frakv) .$$
If $u_0, v_0\in \frakA^+$, then
\begin{equation}
u_0=u_{0,a} \ot u_{0, v}=u_{0,a} u_{0, v}\,\text{ and }\, v_0=v_{0,a} \ot v_{0, v}=v_{0,a} v_{0, v}
\mlabel{eq:u0v04}
\end{equation}
for some $u_{0, a}, v_{0, a}\in \basf$ and $u_{0,v}, v_{0,v}\in A^+$. We get
\begin{align*}
\bar{f}(\fraku\rsham \frakv)&=\bar{f}(u_0 v_0)=\bar{f}(u_{0,a}v_{0,a}(u_{0,v}v_{0,v}))\\
&=u_{0,a}v_{0,a}f(u_{0,v}v_{0,v})=u_{0,a}f(u_{0,v})v_{0,a}f(v_{0,v})\\
&=\bar{f}( u_0)\bar{f}( v_0)=\bar{f}(\fraku)\bar{f}(\frakv).
\end{align*}

Suppose that Eq.~(\mref{eq:comf}) has been verified for $m+n \leq k$ with $k\geq 0$, and consider the case of $m+n = k+1\geq 1$.
There are four cases to consider.

\noindent{\bf Case 1.} $u_0\in \basf$ and $v_0\in \basf$. In this case, we have
\begin{align*}
&\ \bar{f}(\fraku\rsham \frakv)\\
=&\ \bar{f}\left(\left(u_0\ot \twovec{\omega_1}{u_1}\tot \bar{\frak u}'\right) \rsham \left(v_0 \ot \twovec{\bo_1}{v_1}\tot \bar{\frak v}'\right)\right)\\
=&\ \bar{f}\left(u_0 v_0 \ot \left(\left( \twovec{\omega_1}{u_1}\tot \bar{\frak u}'\right) \shap \left( \twovec{\bo_1}{v_1}\tot \bar{\frak v}'\right)\right)\right)\quad (\text{by Eq.~(\mref{eq:matdia})})\\
=&\ \bar{f}\left(u_0 v_0 \ot \twovec{\omega_1}{u_1} \tot \left(\bar{\frak u}'\shap \left( \twovec{\bo_1}{v_1}\tot \bar{\frak v}'\right) \right)\right)+ \bar{f}\left(u_0 v_0 \ot \twovec{\bo_1}{v_1}\tot \left(\left( \twovec{\omega_1}{u_1}\tot \bar{\frak u}'\right)\shap \bar{\frak v}'\right)\right)\quad (\text{by Eq.~(\mref{eq:shrec})})\\
=&\ u_0v_0 P_{\omega_1, R}\bar{f}\left(u_1\ot \left(\bar{\fraku}'\shap \left( \twovec{\bo_1}{v_1}\tot \bar{\frak v}'\right)\right)\right)
+u_0v_0 P_{\bo_1, R}\bar{f}\left( v_1\ot \left( \left( \twovec{\omega_1}{u_1}\tot \bar{\frak u}'\right) \shap \bar{\frakv}'\right)\right)\quad (\text{by Eq.~(\mref{eq:dfu})})\\
=&\ u_0v_0 P_{\omega_1, R}\bar{f}\left((u_1\ot \bar{\fraku}')\rsham  \left( 1\ot \twovec{\bo_1}{v_1}\tot \bar{\frak v}'\right)\right)
+u_0v_0 P_{\bo_1, R}\bar{f}\left(  \left( 1\ot \twovec{\omega_1}{u_1}\tot \bar{\frak u}'\right)\rsham (v_1\ot  \bar{\frakv}')\right)\\
&\hspace{8cm} (\text{by Eq.~(\mref{eq:matdia})})\\
=&\ u_0v_0 P_{\omega_1, R}\bar{f}(u_1\ot \bar{\fraku}')\bar{f} \left( 1\ot \twovec{\bo_1}{v_1}\tot \bar{\frak v}'\right)
+u_0v_0 P_{\bo_1, R}\bar{f}  \left( 1\ot \twovec{\omega_1}{u_1}\tot \bar{\frak u}'\right)\bar{f} (v_1\ot  \bar{\frakv}')\\
&\hspace{8cm}(\text{by the induction hypothesis})\\
=&\ u_0v_0 P_{\omega_1, R}\bar{f}(u_1\ot \bar{\fraku}')\bar{f} ( P_{\bo_1}(v_1\ot \bar{\frak v}'))
+u_0v_0 P_{\bo_1, R}\bar{f}(P_{\omega_1}({u_1}\ot \bar{\frak u}'))\bar{f} (v_1\ot  \bar{\frakv}')\quad (\text{by Eq.~(\mref{eq:relrbo})})\\
=&\ u_0v_0 P_{\omega_1, R}\Big(\bar{f}(u_1\ot \bar{\fraku}')P_{\bo_1, R}(\bar{f}(v_1\ot \bar{\frak v}'))\Big)
+u_0v_0P_{\bo_1, R}\Big(P_{\omega_1, R}(\bar{f}(u_1\ot \bar{\fraku}'))\bar{f}(v_1\ot \bar{\frak v}')\Big)\quad (\text{by Eq.~(\mref{eq:cm22})})\\
=&\ u_0v_0 P_{\omega_1, R}(\bar{f}(u_1\ot \bar{\fraku}'))P_{\bo_1, R}(\bar{f}(v_1\ot \bar{\frak v}'))\quad (\text{by $(R, P_{\Omega, R})$ being a \match Rota-Baxter algebra})\\
=&\ \Big(u_0 P_{\omega_1, R}(\bar{f}(u_1\ot \bar{\fraku}'))\Big)\Big(v_0P_{\bo_1, R}(\bar{f}(v_1\ot \bar{\frak v}'))\Big)\quad (\text{by the commutativity})\\
=&\ \bar{f}\left(u_0\ot \twovec{\omega_1}{u_1}\tot \bar{\fraku}'\right)\bar{f}\left(v_0\ot \twovec{\bo_1}{ v_1}\tot \bar{\frak v}'\right)\quad (\text{by Eq.~(\mref{eq:dfu})})\\
=&\ \bar{f}({\fraku})\bar{f}(\frakv).
\end{align*}

\noindent{\bf Case 2.} $u_0\in \basf$ and $v_0\in \frakA^+$. In this case, we have
\begin{align*}
&\ \bar{f}(\fraku\rsham \frakv)\\
=&\ \bar{f}\left(\left(u_0\ot \twovec{\omega_1}{u_1}\tot \bar{\frak u}'\right) \rsham \left(v_0 \ot \twovec{\bo_1}{v_1}\tot \bar{\frak v}'\right)\right)\\
=&\ \bar{f}\left(u_0 v_0 \ot \twovec{\omega_1}{u_1} \tot \left(\bar{\frak u}'\shap \left( \twovec{\bo_1}{v_1}\tot \bar{\frak v}'\right) \right)\right)+ \bar{f}\left(u_0 v_0 \ot \twovec{\bo_1}{v_1}\tot \left(\left( \twovec{\omega_1}{u_1}\tot \bar{\frak u}'\right)\shap \bar{\frak v}'\right)\right)\quad (\text{by Eq.~(\mref{eq:shrec})})\\
=&\ \bar{f}(u_0v_0) P_{\omega_1, R}\bar{f}\left(u_1\ot \left(\bar{\fraku}'\shap \left( \twovec{\bo_1}{v_1}\tot \bar{\frak v}'\right)\right)\right)
+\bar{f}(u_0v_0) P_{\bo_1, R}\bar{f}\left( v_1\ot \left( \left( \twovec{\omega_1}{u_1}\tot \bar{\frak u}'\right) \shap \bar{\frakv}'\right)\right)\\
&\hspace{8cm}  (\text{by Eq.~(\mref{eq:dfuu2})})\\
=&\ \bar{f}(u_0v_0) P_{\omega_1, R}\bar{f}\left((u_1\ot \bar{\fraku}')\rsham  \left( 1\ot \twovec{\bo_1}{v_1}\tot \bar{\frak v}'\right)\right)
+\bar{f}(u_0v_0) P_{\bo_1, R}\bar{f}\left(  \left( 1\ot \twovec{\omega_1}{u_1}\tot \bar{\frak u}'\right)\rsham (v_1\ot  \bar{\frakv}')\right)\\
&\hspace{8cm} (\text{by Eq.~(\mref{eq:matdia})})\\
=&\ \bar{f}(u_0v_0) P_{\omega_1, R}\bar{f}(u_1\ot \bar{\fraku}')\bar{f} \left( 1\ot \twovec{\bo_1}{v_1}\tot \bar{\frak v}'\right)
+\bar{f}(u_0v_0) P_{\bo_1, R}\bar{f}  \left( 1\ot \twovec{\omega_1}{u_1}\tot \bar{\frak u}'\right)\bar{f} (v_1\ot  \bar{\frakv}')\\
&\hspace{8cm}(\text{by the induction hypothesis})\\
=&\ \bar{f}(u_0v_0) P_{\omega_1, R}\bar{f}(u_1\ot \bar{\fraku}')\bar{f} ( P_{\bo_1}(v_1\ot \bar{\frak v}'))
+\bar{f}(u_0v_0) P_{\bo_1, R}\bar{f}(P_{\omega_1}({u_1}\ot \bar{\frak u}'))\bar{f} (v_1\ot  \bar{\frakv}')\quad (\text{by Eq.~(\mref{eq:relrbo})})\\
=&\ \bar{f}(u_0v_0) P_{\omega_1, R}\Big(\bar{f}(u_1\ot \bar{\fraku}')P_{\bo_1, R}(\bar{f}(v_1\ot \bar{\frak v}'))\Big)
+\bar{f}(u_0v_0)P_{\bo_1, R}\Big(P_{\omega_1, R}(\bar{f}(u_1\ot \bar{\fraku}'))\bar{f}(v_1\ot \bar{\frak v}')\Big)\\
&\hspace{8cm} (\text{by Eq.~(\mref{eq:cm22})})\\
=&\ \bar{f}(u_0v_0) P_{\omega_1, R}(\bar{f}(u_1\ot \bar{\fraku}'))P_{\bo_1, R}(\bar{f}(v_1\ot \bar{\frak v}'))\quad (\text{by $(R, P_{\Omega, R})$ being a \match Rota-Baxter algebra})\\
=&\ u_0\bar{f}(v_0) P_{\omega_1, R}(\bar{f}(u_1\ot \bar{\fraku}'))P_{\bo_1, R}(\bar{f}(v_1\ot \bar{\frak v}'))\quad (\text{by $\bar{f}$ being $(\basf, \kappa_\Omega)$-linearity})\\
=&\ \Big(u_0 P_{\omega_1, R}(\bar{f}(u_1\ot \bar{\fraku}'))\Big)\Big(\bar{f}(v_0)P_{\bo_1, R}(\bar{f}(v_1\ot \bar{\frak v}'))\Big)\quad (\text{by the commutativity})\\
=&\ \bar{f}({\fraku})\bar{f}(\frakv)\quad (\text{by Eq.~(\mref{eq:dfu}) and ~(\ref{eq:dfuu2})}).
\end{align*}

\noindent{\bf Case 3.} $u_0\in \frakA^+$ and $v_0\in \basf$. This case is similar to Case 2.

\noindent{\bf Case 4.} $u_0\in \frakA^+$ and $v_0\in \frakA^+$. In this case, we write $u_0$ and $v_0$
in the form of Eq.~(\mref{eq:u0v04}) and obtain
\begin{align*}
\fraku=&\ u_0\ot \twovec{\omega_1}{u_1}\tot \bar{\fraku}' =u_0 \rsham P_{\omega_1}({u_1}\ot \bar{\fraku}')\in \frakA\ot (\bfk\Omega \ot\frakA^+)^{\ot m},\\
\frakv=&\ v_0 \ot \twovec{\bo_1}{v_1}\tot \bar{\frak v}' =v_0 \rsham P_{\bo_1}({v_1}\ot \bar{\frak v}')\in \frakA\ot (\bfk\Omega \ot\frakA^+)^{\ot n},
\end{align*}
which implies
\begin{align*}
&\ \bar{f}(\fraku\rsham \frakv)\\
=&\ \bar{f}\Big( (u_{0}\rsham P_{\omega_1}({u_1}\ot \bar{\fraku}')) \rsham ( v_{0}\rsham P_{\bo_1 }({v_1}\ot \bar{\frak v}')) \Big) \\
=&\ \bar{f}\left((u_0v_0)\rsham P_{\omega_1}({u_1}\ot \bar{\fraku}')\rsham
P_{\bo_1}({v_1}\ot \bar{\frak v}')\right)\quad(\text{by $\rsham$ being commutative})\\
=&\ \bar{f}\Bigg( (u_0v_0) \rsham \bigg(P_{\omega_1}\Big(({u_1}\ot \bar{\fraku}')\rsham P_{\bo_1}({v_1}\ot \bar{\frak v}')\Big)
+P_{\bo_1}\Big(P_{\omega_1}({u_1}\ot \bar{\fraku}')\rsham ({v_1}\ot \bar{\frak v}')\Big)
\bigg)\Bigg)
\quad(\text{by Step 1})\\
=&\ \bar{f}\Bigg( u_0v_0  P_{\omega_1}\Big(({u_1}\ot \bar{\fraku}')\rsham P_{\bo_1}({v_1}\ot \bar{\frak v}')\Big)
+ u_0v_0 P_{\bo_1}\Big(P_{\omega_1}({u_1}\ot \bar{\fraku}')\rsham ({v_1}\ot \bar{\frak v}')\Big)
\Bigg)\\
=&\ \bar{f}(u_0v_0)(P_{\omega_1,R}  \bar{f} )\left({u_1}\ot \bar{\fraku}'\rsham(P_{\bo_1}({v_1}\ot \bar{\frak v}'))\right)
+\bar{f}(u_0v_0)( P_{\bo_1,R} \bar{f})\Big((P_{\omega_1}({u_1}\ot \bar{\fraku}'))\rsham {v_1}\ot \bar{\frak v}'\Big) \\
&\hspace{8cm}(\text{by Eq.~(\mref{eq:dfuu2})})\\
=&\ \bar{f}(u_0v_0)P_{\omega_1, R} \Big( \bar{f} ({u_1}\ot \bar{\fraku}')\bar{f} (P_{\bo_1}({v_1}\ot \bar{\frak v}'))\Big)
+\bar{f}(u_0v_0)P_{\bo_1, R} \Big(\bar{f}(P_{\omega_1}({u_1}\ot \bar{\fraku}')) \bar{f}({v_1}\ot \bar{\frak v}')\Big)
\\
&\hspace{8cm}(\text{by the induction hypothesis})\\
=&\ \bar{f}(u_0v_0)\bigg(P_{\omega_1, R} \Big(\bar{f}({u_1}\ot \bar{\fraku}')P_{\bo_1,R} (\bar{f}({v_1}\ot \bar{\frak v}'))\Big)
+P_{\bo_1,R} \left(P_{\omega_1,R} (\bar{f}({u_1}\ot \bar{\fraku}'))\bar{f}({v_1}\ot \bar{\frak v}')\right)
\bigg)
\ \ (\text{by Eq.~(\ref{eq:cm22}}))\\
=&\ \bar{f}(u_0v_0)P_{\omega_1,R} (\bar{f}({u_1}\ot \bar{\fraku}'))P_{\bo_1,R} (\bar{f}({v_1}\ot \bar{\frak v}'))\quad (\text{by $(R, P_{\Omega, R})$ being a \match Rota-Baxter algebra})\\
=&\ \bar{f}(u_0)\bar{f}(v_0)P_{\omega_1,R} (\bar{f}({u_1}\ot \bar{\fraku}'))P_{\bo_1,R} (\bar{f}({v_1}\ot \bar{\frak v}'))\quad
(\text{by the initial step of this subcase })\\
=&\ \Big(\bar{f}(u_0)P_{\omega_1,R} (\bar{f}({u_1}\ot \bar{\fraku}'))\Big)\Big(\bar{f}(v_0)P_{\bo_1,R} (\bar{f}({v_1}\ot \bar{\frak v}'))\Big)\\
=&\ \bar{f}(\fraku)\bar{f}(\frakv) \quad(\text{by Eq.~(\mref{eq:dfuu2}})).
\end{align*}

({\bf The uniqueness}). Since $\bar{f}$ is a matching $(\basf, \kappa_\Omega)$-Rota-Baxter algebra homomorphism with $f = \bar{f}  j_A$,
it must be determinate uniquely by Eq.~(\mref{eq:ini}) -- Eq.~(\mref{eq:dfuu2}).

\section{\Match dendriform algebras and \match Zinbiel algebras}
\mlabel{sec:mzinb}
In this section, we introduce the concept of commutative matching dendriform algebras and the equivalent notion of \match Zinbiel algebras. We then establish their relationship with \mrbas, generalizing the connection of Zinbiel algebras with commutative Rota-Baxter algebras. Finally, free \match Zinbiel algebras are constructed.

\subsection{Commutative \match dendriform algebras}
Motivated by the natural connection of Rota-Baxter algebras (of weight zero) with dendriform algebras~\mcite{Agu00} on the one hand, and the connection with pre-Lie algebras on the other,
dendriform algebras have been generalized to \match dendriform algebras in~\mcite{ZGG}. We recall this notion and basic properties.

\begin{defn}
\mlabel{de:mdend}
Let $\Omega$ be a \nes. A {\bf \match dendriform algebra} is a module $D$ together with a family of binary operations $(\prec_\omega, \succ_\omega)_{\omega\in \Omega}$, such that, for $ x, y, z\in T$ and $\alpha,\beta\in \Omega$,
\begin{align}
(x\prec_{\alpha} y) \prec_{\beta} z=\ & x \prec_{\alpha} (y\prec_{\beta} z)
+x\prec_{ \beta} (y \succ_{\alpha} z), \mlabel{eq:ddf1} \\
(x\succ_{\alpha} y)\prec_{\beta} z=\ & x\succ_{\alpha} (y\prec_{\beta} z),\quad \quad \quad \quad \quad \ \ \ \ \mlabel{eq:ddf2} \\
 (x\prec_{\beta} y)\succ_{ \alpha}z
 +(x\succ_{\alpha} y) \succ_{\beta}z =\ & x\succ_{\alpha}(y \succ_{\beta} z). \mlabel{eq:ddf3}
\end{align}
\end{defn}
\begin{remark}
\mlabel{rem:md}
\begin{enumerate}
\item Let $(D, \, (\prec_\omega, \succ_\omega)_{\omega\in \Omega})$ be a \match dendriform algebra. Consider linear combinations
\begin{align}
\prec_A:=\sum_{\omega\in \Omega} a_\omega \prec_\omega\,  \text{ and }\,  \succ_A:=\sum_{\omega\in \Omega} a_\omega \succ_\omega,
\, a_\omega \in \bfk,\mlabel{eq:mdo}
\notag
\end{align}
with a finite support. Then $(D, \prec_A, \succ_A)$ is a dendriform algebra.

\item \label{it:md2}
\label{it:RBTD2} An \mrba $(R, \,(P_{\omega})_{\omega \in \Omega} )$ of weight
$\lambda_\Omega = (\lambda_\omega)_{\omega\in \Omega}$ induces a \match dendriform algebra
$(R, \, (\prec_{\omega}, \succ_{\omega})_{\omega \in \Omega})$, where
\begin{equation*}
x\prec_{\omega}y := xP_{\omega}(y)+\lambda_\omega xy\,\text{ and }\,  x\succ_{\omega}y := P_{\omega}(x)y \,\tforall\, x,y \in R, \omega\in \Omega.
\end{equation*}
\end{enumerate}
\end{remark}
Generalizing commutative dendriform algebras, we give
\begin{defn}
Let $\Omega$ be a nonempty set.
A \match dendriform  algebra $(D, (\prec_\omega, \succ_\omega)_{\omega\in \Omega})$ is called {\bf commutative} if
\begin{equation*}
x\succ_{\omega}y=y\prec_{\omega} x \, \tforall\,  x,y \in D, \omega\in \Omega.
\mlabel{eq:comu1}
\end{equation*}
\end{defn}

We give an example of \match dendriform algebras from Volterra integral operators.

\begin{exam}
As in Example~\mref{ex:Zin3}, consider the $\RR$-algebra $R:=\mathrm{Cont}(\RR)$ and a family $ (k_\omega(x))_{\omega\in \Omega}$ of continuous functions on $R$. Define a {\bf Volterra integral operator} $I_\omega: R\longrightarrow R$ by taking
\begin{equation*}
I_\omega(f(x)):= \int_0^x k_\omega(t)f(t)\,dt \quad \text{for } f\in R, \omega \in \Omega.
\end{equation*}
By Example~\mref{ex:Zin3}, $(R, (I_\omega)_{\omega\in \Omega})$ is a commutative \mrba of weight 0. Applying Remark~\mref{rem:md}~(\mref{it:md2}), the operations
\begin{align}
f(x)\prec_\omega g(x):= f(x) \int_0^x k_\omega(t)g(t)\,dt,\,\,
f(x)\succ_\omega g(x):= g(x) \int_0^x k_\omega(t)f(t)\,dt   \, \text{ for }\, f, g \in R, \omega\in \Omega.\mlabel{eq:mzin4}
\notag
\end{align}
equip $R$ with a commutative \match dendriform algebra structure.
\mlabel{exam:Zin}
\end{exam}

The commutative dendriform algebra is equivalent to the Zinbiel algebra which arose as the Koszul dual to a Leibniz algebra
introduced by Loday~\mcite{Lod95}. Free Zinbiel algebras were shown to be precisely the shuffle product algebra~\mcite{Losh}.

\begin{defn}
A {\bf  (left) \mza}  is a module $Z$ together with a family of binary operations
$(\circ_\omega)_{\omega\in \Omega}$ such that, for $x, y, z\in Z$ and $\alpha,\beta\in \Omega$,
\begin{equation}
(x\circ_{\alpha}y)\circ_{\beta}z=x\circ_{\alpha}(y\circ_{\beta}z)
+x\circ_{\beta}(z\circ_{\alpha}y).
\mlabel{eq:lzin}
\end{equation}
\end{defn}

\begin{remark}
\begin{enumerate}
\item Any  Zinbiel algebra can be viewed as a  \mza by taking $\Omega$ to be a singleton.

\item In a \mza $(Z, (\circ_\omega)_{\omega\in \Omega})$, $(Z, \circ_\omega)$ is a  Zinbiel algebra for any $\omega\in \Omega$.
\mlabel{it:zinbb}

\item For a \mza, it follows from Eq.~(\mref{eq:lzin}) that
\begin{align*}
(x\circ_\beta z)\circ_\alpha y=x\circ_\beta (z\circ_\alpha y)+x \circ_\alpha (y\circ_\beta z).
\end{align*}
Note that the right hand side is invariant when $(y,\alpha)$ is replaced by $(z,\beta)$. Thus
$$(x\circ_\alpha y)\circ_\beta z=(x\circ_\beta z)\circ_\alpha y,$$
which is precisely the second axiom of {\bf multiple permutative algebras} introduced by Foissy~\cite[Proposition~12]{Foi18}.
\end{enumerate}
\mlabel{re:zinb}
\end{remark}

It is well-known that a Zinbiel algebra is equivalent to a commutative dendriform algebra~\mcite{Agu00},
which we now generalize to the matching context.

\begin{prop}\label{prop:mda} Let $\Omega$ be a nonempty set.
\begin{enumerate}
\item If $(D, (\prec_{\omega}, \succ_{\omega})_{\omega \in \Omega})$ is a commutative \match dendriform  algebra, then
$(D,(\prec_\omega)_{\omega\in \Omega})$ is a match Zinbiel algebra.
\mlabel{it:comdd}

\item Conversely, if $(D, (\circ_\omega)_{\omega\in \Omega})$ is a \match Zinbiel algebra, define
\begin{align}
y\succ_{\omega}x:= x\prec_{\omega}y:=x\circ_\omega y \, \text{ for }\, x, y\in D, \omega\in \Omega.\mlabel{eq:newd}
\end{align}
Then the pair $(D, (\prec_{\omega}, \succ_{\omega})_{\omega \in \Omega})$ is a commutative \match dendriform algebra.
\mlabel{it:comddd}
\end{enumerate}
\end{prop}

Continuing Example~\mref{exam:Zin}, the pair $(R,\, (\circ_\omega)_{\omega \in \Omega})$
is a \mza with
\begin{align}
f(x)\circ_\omega g(x):= f(x) \int_0^x k_\omega(t)g(t)\,dt   \, \text{ for }\, f, g \in R, \omega\in \Omega.
\notag
\end{align}

\begin{proof}
~(\mref{it:comdd}).
Suppose that $(D, (\prec_{\omega}, \succ_{\omega})_{\omega \in \Omega})$ is a commutative \match dendriform  algebra. Then for $x, y, z\in D$ and $\alpha, \beta \in \Omega$, we have
\begin{align*}
(x\prec_{\alpha}y)\prec_{\beta}z=&\ x \prec_{\alpha} (y\prec_{\beta} z)
+x\prec_{\beta} (y \succ_{\alpha} z) \quad(\text{by Eq.~(\mref{eq:ddf1}}))\\
=&\ x\prec_{\bim{\alpha}}(y\prec_{\beta}z)
+x\prec_{\bim{\beta}}(z\prec_{\alpha}y).
\end{align*}
~(\mref{it:comddd}).
It suffices to prove Eqs.~(\mref{eq:ddf1})--(\mref{eq:ddf3}).
Suppose that $(D, (\circ_\omega)_{\omega\in \Omega})$ is a \match Zinbiel algebra.
Then for $x, y, z\in D$ and $\alpha, \beta \in \Omega$, we have
\begin{equation}
\begin{aligned}
(x\prec_{\alpha}y)\prec_{\beta}z=&\ x\prec_{\alpha}(y\prec_{\beta}z)
+x\prec_{\beta}(z\prec_{\alpha}y)\quad (\text{by Eq.~(\mref{eq:lzin})})\\
=&\ x \prec_{\alpha} (y\prec_{\beta} z)
+x\prec_{\beta} (y \succ_{\alpha} z)\quad(\text{by Eq.~(\mref{eq:newd}})).
\mlabel{eq:ax1}
\end{aligned}
\end{equation}
Also,
\begin{align*}
(x\succ_{{\alpha}}y)\prec_{{\beta}}z=\ &(y\prec_{{\alpha}}x)\prec_{{\beta}}z
=y\prec_{\alpha}(x\prec_{\beta}z)
+y\prec_{\beta}(z\prec_{\alpha}x)\quad(\text{by Eq.~(\mref{eq:lzin}}))\\
=\ &y\prec_{\alpha}(z\succ_{\beta}x)
+y\prec_{\beta}(z\prec_{\alpha}x)\quad(\text{by Eq.~(\mref{eq:newd}}))\\
=\ &(y\prec_{{\beta}}z)\prec_{{\alpha}}x \quad(\text{by Eq.~(\mref{eq:ax1}}))\\
=\ &x\succ_{{\alpha}}(y\prec_{{\beta}}z)\quad(\text{by Eq.~(\mref{eq:newd}})).
\end{align*}
Further,
\begin{align*}
x\succ_{{\alpha}}(y\succ_{{\beta}}z)=\ &(z\prec_{{\beta}}y)\prec_{{\alpha}}x
=z\prec_{\beta}(y\prec_{{\alpha}}x)
+z\prec_{\alpha}(x\prec_{{\beta}}y)\quad(\text{by Eq.~(\mref{eq:lzin}}))\\
=\ &(y\prec_{{\alpha}}x)\succ_{\beta}z
+(x\prec_{{\beta}}y)\succ_{\alpha}z\quad(\text{by Eq.~(\mref{eq:newd}}))\\
=\ &(x\succ_{{\alpha}}y)\succ_{\beta}z
+(x\prec_{{\beta}}y)\succ_{\alpha}z\quad(\text{by Eq.~(\mref{eq:newd}}))\\
=\ & (x\prec_{\beta} y)\succ_{ \alpha}z
 +(x\succ_{\alpha} y) \succ_{\beta}z,
\end{align*}
as required.
\end{proof}

The following result shows that a \mza gives rise to other \mzas by arbitrary finite linear combinations.

\begin{prop}
Let $\Omega$ and $I$ be  nonempty sets  and let $A_i:\Omega\to \bfk, i\in I$, be a family of maps
with finite supports, identified with $A_i= (a_{i,\omega})_{\omega\in \Omega}$.
Let $(Z,\, (\circ_\omega)_{\omega \in \Omega})$ be a \mza.  Consider the linear combinations
\begin{align}
\circ_i:=\circ_{A_i}:=\sum_{\omega\in \Omega} a_{i,\omega} \circ_\omega\, \quad \text{  for }\,  i\in I.
\mlabel{eq:dfav}
\end{align}
Then
$(Z,\, (\circ_i)_{i \in I})$
is a \mza. In particular, $(Z, \circ_i)$ is a \za for each $\omega\in \Omega$.
\end{prop}

\begin{proof}
For $x, y, z\in Z$ and $\alpha, \beta \in \Omega$, we have
\begin{align*}
(x\circ_{i} y) \circ_{j} z
=&\ \sum_{\beta \in \Omega} a_{j, \beta} \left(\sum_{\alpha\in \Omega} a_{i,\alpha} x \circ_\alpha y\right) \circ_{\beta} z \quad \text{(by Eq.~(\mref{eq:dfav}))}\\
=&\ \sum_{\alpha\in \Omega}\sum_{\beta \in \Omega} a_{i,\alpha} a_{j, \beta} (x \circ_\alpha y)\circ_{\beta} z\\
=&\ \sum_{\alpha\in \Omega}\sum_{\beta \in \Omega} a_{i,\alpha} a_{j, \beta} \Big( x \circ_{\alpha} (y\circ_{\beta} z)
+x\circ_{ \beta} (y \circ_{\alpha} z)\Big)\quad \text{(by Eq.~(\mref{eq:lzin}))}\\
=&\ \sum_{\alpha\in \Omega}\sum_{\beta \in \Omega} a_{i,\alpha} a_{j, \beta} x \circ_{\alpha} (y\circ_{\beta} z)
+\sum_{\alpha\in \Omega}\sum_{\beta \in \Omega} a_{i,\alpha} a_{j, \beta} x\circ_{ \beta} (y \circ_{\alpha} z)\\
=&\ \sum_{\alpha\in \Omega} a_{i,\alpha} x \circ_{\alpha} \left(\sum_{\beta \in \Omega}a_{j, \beta} y\circ_{\beta} z\right)
+\sum_{\beta \in \Omega} a_{j, \beta} x\circ_{ \beta} \left(\sum_{\alpha\in \Omega}a_{i,\alpha} y \circ_{\alpha} z\right)\\
=&\ \sum_{\alpha\in \Omega} a_{i,\alpha} x \circ_{\alpha} \left(  y\circ_{j} z\right)
+\sum_{\beta \in \Omega} a_{j, \beta} x\circ_{ \beta} \left( y \circ_{i} z\right) \quad \text{(by Eq.~(\ref{eq:dfav}) )}\\
=&\ x \circ_{i} (y\circ_{j} z)
+x\circ_{ j} (y \circ_{i} z).
\end{align*}
Thus $(Z,\, (\circ_i)_{i \in \Omega})$ is a \mza. Remark~\mref{re:zinb}~(\mref{it:zinbb}) gives the second statement.
\end{proof}

\subsection{The free \match Zinbiel algebra}
In this subsection, we construct the free commutative \match Zinbiel algebra on a module.

\begin{defn}
Let $\Omega$ be a nonempty set and $M$ a module.
A {\bf free \match Zinbiel algebra} on $M$ is a \match Zinbiel algebra $(D, (\circ_{\omega})_{\omega\in \Omega})$
together with a module homomorphism $j_M: M\rightarrow D$ satisfying the universal property:
for any  \match Zinbiel algebra $(D', (\circ_{\omega,\,D'})_{\omega\in \Omega})$ and any module homomorphism $f: M\rightarrow D',$
there is a unique \match Zinbiel algebra homomorphism $\bar{f}: D\rightarrow D'$ such that $f=\bar{f}  j_M.$
\end{defn}

For a module $M$, we can regard $M$ as a commutative algebra equipped with the zero multiplication.
By Theorem~\mref{thm:comfree}, the free commutative \mrba of weight zero on $M$ is
\begin{equation*}
\msha(M):= M\otimes\Shu(M_\Omega)=M \bigoplus \left(\bigoplus_{k\geq 1}   M \ot (\bfk \Omega \ot M)^{\otimes k}\right)
\end{equation*}
equipped with the product $\diamond$ defined in Eq.~\meqref{eq:matdia} and operators defined in Eq.~\meqref{eq:op}.
Let
\begin{align*}
\fraka:=&\  a_0\ot \frak{\bar a} :=a_0\ot \twovec{\alpha_1}{a_1} \tot \bar{\fraka}'  := a_0\ot \twovec{\alpha_1}{a_1} \tot \twovec{\alpha_2}{a_2}\tot   \cdots  \tot \twovec{\alpha_m}{a_m}\in \sha(M_\Omega),\\
\frakb:=&\ b_0\ot \frak{\bar b} :=\ b_0\ot \twovec{\beta_1}{b_1}\tot \bar{\frakb}' := b_0\ot \twovec{\beta_1}{b_1}\tot \twovec{\beta_2} {b_2} \tot \cdots  \tot  \twovec{\beta_n} {b_n} \in \sha(M_\Omega)
\end{align*}
with $m, n\geq 0$. For $\omega\in \Omega$, define
\begin{equation}
\begin{aligned}
\prec_{\omega}:\ &\msha(M) \ot \msha(M) \rightarrow \msha(M),\\
&\fraka\ot \frakb \mapsto a_0 \ot  \left(\bar{\fraka} \shap\left(\twovec{\omega}{b_0} \tot \bar{\frakb} \right)\right)
= \fraka \rsham \left(1\ot \twovec{\omega}{b_0}\tot \bar{\frakb} \right) = \fraka\rsham P_{\omega}(\frakb).
\end{aligned}
\mlabel{eq:xiaoyu}
\end{equation}

Now we are ready to give the free matching Zinbiel algebras.

\begin{theorem}
Let $\Omega$ be a nonempty set and $M$ a module.
Then $(\msha(M), (\prec_{\omega})_{\omega \in \Omega})$,
together with the natural algebra homomorphism $j_{M}: M \rightarrow \msha(M)$,
is the free \match Zinbiel algebra on $M$.
\mlabel{thm:comu}
\end{theorem}

\begin{remark}\mlabel{rk:shuzin}
\begin{enumerate}
	\item
On the free commutative Rota-Baxter algebra of weight zero $\sha(M):=M\otimes \Shu(M)$, the multiplication
\begin{equation}\mlabel{eq:double}
\fraka \star \frakb: = \fraka \prec \frakb + \fraka \succ \frakb = \fraka \rsham P(\frakb) + P(\fraka) \sham \frakb \quad \text{ for } \fraka, \frakb\in \sha(M),
\end{equation}
is precisely the shuffle product $\shap$ defined in Eq.~\meqref{eq:shrec}, noting that $\sha(M)= \Shu(M)^+$ as modules~\mcite{Gub}. Here $\Shu(M)=\bfk \oplus \Shu(M)^+$. Thus when $\Omega$ is a singleton, the theorem recovers the characterization of Loday~\mcite{Lod95} that the nonunitary shuffle product algebra $\Shu(M)^+$ on $M$ is the free Zinbiel algebra on $M$.
\item
It is also worth noticing that, for general $\Omega$, only the nonunitary shuffle product algebra $\Shu(\bfk\Omega\ot M)^+$ is not enough to give a \match Zinbiel algebra via Eq.~\meqref{eq:xiaoyu}, not to mention the free \match Zinbiel algebra on $M$. This is because a factor $b_0$ in $M$ is required to define the products $\prec_\omega$ in Eq.~\meqref{eq:xiaoyu}. To see this in another way, $\msha(M)=M\ot \Shu(M_\Omega)$ is linearly identified with $\Shu(M_\Omega)^+$ only when $\Omega$ is a singleton. The above connection of the shuffle product algebra with free Zinbiel algebras in Loday's work is thanks to the fact that the shuffle product algebra $(\Shu(M)^+,\shap\!)$ happens to be the algebra $(\sha(M),\star)$ with the derived product $\star$ in Eq.~\meqref{eq:double}.
\end{enumerate}
\end{remark}

\begin{proof}
By Theorem~\mref{thm:comfree}, the triple $(\sha(M_\Omega), \sham, (P_{\omega,\, M})_{\omega\in \Omega})$ is the free commutative \match Rota-Baxter algebra of weight zero on $M$. So $(\sha(M_\Omega), (\prec_{\omega}, \succ_{\omega})_{\omega\in \Omega})$ is a \match dendriform  algebra by Remark~\mref{rem:md}~(\mref{it:md2}), where
\begin{equation}
\begin{aligned}
\fraka\succ_{\omega}\frakb : =& P_\omega(\fraka) \sham \frakb
= \left(1\ot \twovec{\omega}{a_0}\tot \bar{\fraka}\right)
\sham \frakb=\ b_0 \ot\left(\left( \twovec{\omega}{a_0}\tot \bar{\fraka} \right)\shap  \bar{\frakb} \right), \\
\fraka\prec_{\omega}\frakb:=&\fraka\sham P_\omega(\frakb)= \fraka \sham \left(1\ot \twovec{\omega}{b_0}\tot \bar{\frakb}\right) =a_0 \ot \left(  \bar{\fraka} \shap \left( \twovec{\omega}{b_0} \tot \bar{\frakb} \right)\right).
\end{aligned}
\mlabel{eq:defd}
\end{equation}
Further, it is commutative, as $(\sha(M_\Omega), \sham)$ is commutative. Thus it follows from Proposition~\mref{prop:mda}~(\mref{it:comdd})  that $(\sha(M_\Omega), (\prec_{\omega})_{\omega \in \Omega})$ is a matching Zinbiel algebra.

We next prove that $(\sha(M_\Omega), (\prec_{\omega})_{\omega \in \Omega}, )$ is the free matching Zinbiel algebra on $M$ by verifying
its universal property. For this, let $(D, (\prec_{\omega,\,D})_{\omega \in \Omega})$ be a matching Zinbiel algebra and let $f: M\rightarrow D$ be a module homomorphism. We prove that there is unique homomorphism $\free{f}:\sha(M_\Omega)\rightarrow D$ of matching Zinbiel algebras such that $\free{f} j_M=f$.

({\bf The existence}). Define a linear map $\bar{f}: \sha(M_\Omega)\rightarrow D$ as follows. For $\fraka=a_0\ot \twovec{\alpha_1}{a_1} \tot \bar{\fraka}'\in M\ot (\bfk \Omega \ot M)^{\ot m}$ with $a_0\in M$,  and $\twovec{\alpha_1}{a_1} \tot \bar{\fraka}'=\bar{\fraka}\in (\bfk \Omega \ot M)^{\ot m}$, we define $\bar{f}(\fraka)$ by induction on $m\geq 0$. For the initial step of $m=0,$ we have $\fraka=a_0$ and define
\begin{align}
\bar{f}(\fraka)=\bar{f}(a_0):=f(a_0). \mlabel{eq:init}
\end{align}
For the induction step of $m\geq 1,$ we define
\begin{equation}
\bar{f}(\mathfrak{ a}):=\bar{f}\left(a_0\ot \twovec{\alpha_1}{a_1} \tot \bar{\fraka}'\right):= f(a_{0})\prec_{{\alpha_1},\,D }\bar{f}({a_1} \ot \bar{\fraka}').
\mlabel{eq:comho}
\end{equation}

We now prove that $\bar{f}$ is a homomorphism of \match dendriform  algebras:
\begin{equation}
\bar{f}(\fraka\prec_{{\omega}}\frakb)=\bar{f}(\mathfrak{a})\prec_{{\omega},\,D}\bar{f}(\frakb)
\mlabel{eq:comff}
\end{equation}
for $\fraka=a_0\ot \twovec{\alpha_1}{a_1} \tot \bar{\fraka}' \in M\ot (\bfk \Omega \ot M)^{\ot m}$ and $\frakb=b_0\ot \twovec{\beta_1}{b_1}\tot \bar{\frakb}' \in M\ot (\bfk \Omega \ot M)^{\ot n}$ with $m,n\geq 0$.

We proceed to prove Eq.~(\mref{eq:comff}) by induction on $m+n\geq 0$. When $m+n=0$, we have $\mathfrak{a}=a_0$ and $\frakb=b_0$ are in $M$. So by Eqs.~(\ref{eq:xiaoyu}), (\ref{eq:init}) and~(\mref{eq:comho}),
\begin{align*}
\bar{f}(\fraka\prec_{{\omega}}\frakb)= \bar{f}(a_0 \prec_\omega b_0) = \bar{f}(a_0\ot \twovec{\omega}{b_0}) =f(a_0)\prec_{{\omega} ,\,D }f(b_0) =\bar{f}(\mathfrak{ a})\prec_{{\omega},\,D}\bar{f}(\frakb).
\end{align*}
Assume that Eq.~\emph{}(\mref{eq:comff}) has been proved when $m+n=k$ for a $k\geq 0$,
and consider the case of $m+n = k+1$.
Then
\begin{align*}
&\ \bar{f}(\fraka\prec_{{\omega}}\frakb)\\
=&\ \bar{f}\left(a_0 \ot \left(  \bar{\fraka} \shap \left( \twovec{\omega}{b_0} \tot \bar{\frakb} \right)\right)\right)\\
=&\bar{f}\Bigg(\ a_0 \ot \twovec{\alpha_1}{a_1}\tot  \left( \bar{\fraka}'  \shap   \left(\twovec{\omega} {b_0} \tot \bar{\frakb}\right) \right)
+a_0 \ot \twovec{\omega}{b_0} \tot \left( \left( \twovec{\alpha_1}{a_1}\tot \bar{\fraka}'\right) \shap  \bar{\frakb}\right)\Bigg) \quad (\text{by Eq.~(\ref{eq:shrec})}) \\
=&\ \bar{f}(a_0)\prec_{{\alpha_1},\, D}\bar{f} (({a_1} \ot \bar{\fraka}')\prec_{{\omega}}\frakb)+\bar{f}(a_0)\prec_{{\omega},\,D }\bar{f} (\frakb\prec_{{\alpha_1}}({a_1} \ot \bar{\fraka}')) \quad(\text{by Eqs.~(\ref{eq:defd}) and~(\ref{eq:comho})})\\
=&\ \bar{f}(a_0)\prec_{{\alpha_1},\, D} \Big(\bar{f}({a_1} \ot \bar{\fraka}')\prec_{{\omega},\,D}\bar{f} (\frakb)\Big)+ \bar{f}(a_0)\prec_{{\omega},\,D} \Big(\bar{f}(\frakb)\prec_{{\alpha_1},\,D}\bar{f}({a_1} \ot \bar{\fraka}')\Big)\\
&\hspace{4.5cm}\quad(\text{by the induction hypothesis})\\
=&\ (\bar{f}(a_0)\prec_{{\alpha_1},\, D}\bar{f}({a_1} \ot \bar{\fraka}'))\prec_{{\omega},\,D}
\bar{f}(\frakb)\quad(\text{by Eq.~(\mref{eq:lzin})})\\
=&\ ({f}(a_0)\prec_{{\alpha_1},\,D}\bar{f}({a_1} \ot \bar{\fraka}'))\prec_{{\omega},\,D }\bar{f}(\frakb)\quad(\text{by Eq.~(\ref{eq:init}}))\\
=&\ \bar{f}(\fraka)\prec_{{\omega},\, D}\bar{f}(\frakb)\quad(\text{by Eq.~(\mref{eq:comho}})),
\end{align*}

({\bf The uniqueness}). Suppose that $\bar{f}': \sha_\Omega(M)\rightarrow D$ is a homomorphism of  \match Zinbiel  algebras with $\free{f}' j_M=f$.
We verify $\bar{f}'(\fraka)=\free{f}(\fraka)$ for $\fraka=a_0\ot \twovec{\alpha_1}{a_1} \tot \bar{\fraka}'\in M\ot (\bfk \Omega \ot M)^{\ot m}$ by induction on $m\geq 0$. For the initial step of $m=0,$ we have $\fraka=a_0$. By Eq.~(\mref{eq:init}), we have
\begin{align*}
\bar{f}'(\fraka)=\bar{f}'( a_0)=f(a_0)=\free{f}(\fraka).
\end{align*}
For the induction step of $m\geq 1$, we have
$$\bar{f}'(\mathfrak{ a})=\bar{f}'(a_0\ot \twovec{\alpha_1}{a_1} \tot \bar{\fraka}') =\bar{f}'(a_0\prec_{{\alpha_1} } ({a_1} \ot \bar{\fraka}'))
=\bar{f}'(a_0)\prec_{{\alpha_1},\, D}\bar{f}'({a_1} \ot \bar{\fraka}')=f(a_0)\prec_{{\alpha_1},\, D}\bar{f}'({a_1} \ot \bar{\fraka}').$$
By the induction hypothesis, $\bar{f}'({a_1} \ot \bar{\fraka}')=\bar{f}({a_1} \ot \bar{\fraka}')$. Then $\bar{f}'(\fraka)=\free{f}(\fraka)$ by the construction of $\free{f}(\fraka)$. This completes the induction.
\end{proof}

\noindent {\bf Acknowledgments}: This work was supported by the National Natural Science Foundation of China (Grant No.\@ 11771190, 12071191,  11861051) and the Natural Science Foundation of Gansu Province (Grant No.~20JR5RA249). We thank the anonymous referee for very helpful suggestions.

\end{document}